\documentclass[12pt]{article}

\usepackage[top=1in, bottom=1in, left=1in, right=1in]{geometry}

\usepackage[style=trad-abbrv, maxnames=99, maxalphanames=9, doi=false]{biblatex}

\bibliography{ComplexGeometry.bib}
\usepackage{amsmath,amssymb,amsthm}
\usepackage{xcolor}
\usepackage[
  colorlinks=true,
  allcolors=black
]{hyperref}

\newtheorem{theorem}{Theorem}[section]
\newtheorem{proposition}[theorem]{Proposition}

\newtheorem{lemma}[theorem]{Lemma}
\newtheorem{corollary}[theorem]{Corollary}
\theoremstyle{remark}
\newtheorem{remark}[theorem]{Remark}

\newcommand{\vol}{\textup{vol}}
\newcommand{\RR}{\Bbb R}
\newcommand{\PSH}{\operatorname{PSH}}
\newcommand{\Sing}{\operatorname{Sing}}
\newcommand{\Supp}{\operatorname{Supp}}
\newcommand{\ddc}{dd^c}

\def\beq{\begin{equation}}
\def\eeq{\end{equation}}

\title{Demailly--Koll\'ar continuity on klt pairs, and applications to alpha and delta invariants}
\author{Tam\'as Darvas and Kewei Zhang}
\date{}

\begin{document}
\maketitle

\begin{abstract}
We establish a Demailly--Koll\'ar type continuity theorem for
plurisubharmonic functions with respect to adapted measures on normal
complex analytic klt pairs. 

As applications,
we prove the equality of the analytic and divisorial versions of the alpha and delta invariants on compact normal K\"ahler klt
pairs, thereby completing a program initiated by the second author. We further
show that the two local alpha invariants introduced by Guedj and
Trusiani for an isolated log terminal singularity coincide and that
their common value is the $n$th root of Li's normalized volume.

These identities have geometric consequences. The equality for the
delta invariant yields a Yau--Tian--Donaldson type divisorial criterion
for the solvability of twisted K\"ahler--Einstein equations in big
cohomology classes, without a semipositivity assumption on the twist. The local alpha identity
determines algebraically the critical exponent governing the existence of positively curved KE metrics near an isolated log terminal
singularity, thus confirming a prediction of Guedj and Trusiani.
\end{abstract}

\tableofcontents

\section{Introduction}

\paragraph{Background.}
One major theme in pluripotential theory is to understand the singularities of plurisubharmonic (psh) functions, since they are intimately related to the extendability of holomorphic objects, solvability of partial differential equations, birational geometry of algebraic varieties, and stability notions for canonical metrics. 

In this direction of study, Demailly--Koll\'ar proved the lower semi-continuity of complex singularity exponents of psh functions on a complex manifold \cite{DK01}.
This property turned out to be extremely useful in various problems. For instance, it led to the important fact that Tian's analytic $\alpha$ invariant for a polarized K\"ahler manifold coincides with the global log canonical threshold arising from birational geometry (see  \cite[Theorem~A.3]{CS08}, \cite[Theorem~2.2]{Shi10}), which further led to the discovery of numerous examples of K\"ahler--Einstein manifolds (see e.g. \cite{Che08DP,CS08,CS09Extremal,CPW14}). 
Based on a similar philosophy, in \cite{Zh21} the second author proved that a certain optimal Moser--Trudinger exponent of a polarized K\"ahler manifold (known as the analytic $\delta$ invariant \cite{Zhang21}) coincides with the valuative $\delta$ invariant arising from the K-stability theory \cite{FO18,BJ20}, thus proving a Yau--Tian--Donaldson (YTD) correspondence for twisted K\"ahler--Einstein metrics beyond the Fano setting (with the twisting term not necessarily semipositive). Note that the argument of \cite{Zh21} also uses Demailly--Koll\'ar's result at a crucial point (see the proof of \cite[Proposition~3.1]{RTZ20}).

At the heart of the original Demailly--Koll\'ar theorem is an $L^1$ \emph{continuity} statement for exponential psh weights, from which the lower semi-continuity of log canonical thresholds follows. While the original theorem concerns complex manifolds, this algebraic consequence holds in substantially greater generality, allowing the underlying space to be singular (see \cite[\S~8, Lemma~8.6]{Kol97} and \cite[Corollary~1.10]{Ambro16}). This generality is essential in birational geometry, where singularities are inevitable. That said, an extension of the Demailly--Koll\'ar continuity result to singular spaces has remained elusive since the publication of \cite{DK01}. Such an extension has been explicitly sought for in several works, including \cite[p.~2]{Liu13}, \cite[Sub-problem A2]{PanThesis23}, \cite[below Proposition~D]{PT25}, \cite[p.~742]{GT23}, and \cite[Remark~1.6]{KimKollar24}, reflecting continued interest in this question.

A key difficulty in proving a singular Demailly--Koll\'ar continuity theorem is that it does not follow directly from resolution of singularities. Indeed, the adapted measure lifts to a measure on a resolution with weight $e^{\chi-\psi}$, where $\chi,\psi$ are psh, and Demailly--Koll\'ar continuity fails for general weights of this form; see \cite[Remark~1.3]{Hi14}. Overcoming this obstruction requires use of intrinsic properties of the singular space and its adapted measure; see the discussion following Lemma~\ref{lemma: vol_domination_sup}.

The goal of this article is to settle this problem, showing that the pathological phenomenon presented in \cite[Remark 1.3]{Hi14} does not appear when dealing with pullbacks of the adapted measure of a klt pair. We then give several applications primarily to K\"ahler--Einstein problems. In particular, we will extend results regarding $\alpha$ and $\delta$ invariants from \cite{CS08,Shi10,Zhang21,Zh21,DZ24} to the singular setting. 

\paragraph{Demailly--Koll\'ar continuity for klt pairs.}
Let $(X,\Delta)$ be a connected normal complex analytic pair of pure dimension $n$, where
\[
\Delta=\sum_l d_l\Delta_l\geq 0
\]
is an effective analytic (Weil) $\mathbb R$-divisor and
$K_X+\Delta$ is $\mathbb R$-Cartier. Recall that an $\mathbb R$-Weil
divisor $P$ is $\mathbb R$-Cartier if it admits a finite presentation
$
P=\sum_\alpha c_\alpha P_\alpha, \ c_\alpha\in\mathbb R,
$
where each $P_\alpha$ is Cartier, i.e., locally principal.

The pair $(X,\Delta)$ is klt if  there exists a log resolution $\pi: Y \to X$ for $(X,\Delta)$ and an $\Bbb R$-Cartier divisor $A$ on $Y$ such that 
\begin{equation}
\label{eq: discrepancy}
  K_Y=\pi^*(K_X+\Delta)+A,
  \qquad
  A=\sum_j a_jE_j,
\end{equation}
and $a_j>-1$.

We choose an $\mathbb R$-Cartier presentation of $K_X+\Delta = \sum_k d_k D_k$ and  a smooth (formal) Hermitian metric $h_\Delta$ on the associated $\mathbb R$-line bundle. This simply means that $h_\Delta = \Pi_k h^{d_k}_k$, where $h_k$ is a choice of Hermitian metric on $\mathcal O(D_k)$.

On  $X_{\mathrm{reg}}$  a nonvanishing local $\mathbb R$-frame
$\sigma$ of $K_X+\Delta = K_X + \sum_l d_l \Delta_l$ is given by the formal tensor product $\eta \otimes \Pi_l f_l^{-d_l}$, where locally $\Delta_l = \{f_l =0\}$ and $\eta$ is a non-vanishing local $n$-form. Introducing the local $(n,n)$-density
$\sigma\wedge\overline\sigma =  \Pi_l |f_l|^{-2d_l} \eta \wedge \overline{\eta}$ we define the following measure:
\begin{equation}
\label{eq: adapted_measure_def}
 d\mu:=
  \frac{(i)^{n^2}\sigma\wedge\overline\sigma}
       {h_\Delta(\sigma,\sigma)}.
\end{equation}
It is elementary to see that $d\mu$ is well defined on $X_\textup{reg}$, only depends on the choice of $h_\Delta$, but not the choice of $\sigma$. We extend $d\mu$ to $X$, by putting no mass on $X_\textup{Sing}$.

On the log resolution $Y$ this measure has the local form
\begin{equation}
\label{eq: adapted_density_res}
 d \tilde \mu=e^G\prod_i|z_i|^{2a_i}\,dV_Y,
 \qquad \pi_*\tilde \mu=\mu,
\end{equation}
where $G$ is smooth, $dV_Y$ is a smooth volume form,  and the divisors $E_i=\{z_i=0\}$ have simple normal
crossings.  The inequalities $a_i>-1$ imply local integrability of $d \tilde \mu$ on $Y$.

Our normal space $X$ also admits local analytic K\"ahler volumes
$dV_X$. We choose a local embedding
$
X\hookrightarrow \mathbb C^N$. We restrict the Euclidean K\"ahler
metric $\omega_{\mathrm{Euc}}$ of \(\mathbb C^N\) to \(X_{\mathrm{reg}}\).  This will give us $dV_X$ using the following formula:
\begin{equation}\label{eq:  DVX_def}
dV_X
=\frac{\omega_{\mathrm{Euc}}^n}{n!}\bigg|_{X_{\mathrm{reg}}}.
\end{equation}
After declaring $dV_X(X_{\Sing})=0$, $dV_X$ becomes a locally defined measure. Such measures are of course not unique, however given a different local K\"ahler volume $d\tilde V_X$ on the same open set $U \subset X$ (defined using a different local embedding),  there exists $C_y> 1$ such that $1/C_y d\tilde V_X \leq d V_X \leq C_y d\tilde V_X$ on a neighborhood of any $y \in U$.

Let $\Omega\subset X$ be open and let $u_j,u\in\PSH(\Omega)$. We say
that $u_j\to u$ in $L^1_{\mathrm{loc}}(\Omega)$ if the convergence
holds with respect to one, and hence any, local K\"ahler volume
measure $dV_X$.

For a psh function $u$ and $K\Subset\Omega$ define
\begin{equation}
\label{eq: compact_threshold}
 c_{\mu,K}[u]
 :=\sup\left\{
 c>0:\ e^{-cu}\in L^1(U,\mu)
 \text{ for some }K\subset U\Subset\Omega
 \right\}.
\end{equation}

Our first result extends the Demailly--Koll\'ar continuity theorem \cite{DK01} to the klt setting.

\begin{theorem}[klt Demailly--Koll\'ar continuity]
\label{thm: main_klt}
Let $(X,\Delta)$ be a normal analytic klt $\mathbb R$-pair with
$\Delta\geq 0$, let $\mu$ be an adapted measure, and let
$K\Subset\Omega\subset X$. Suppose
$ u_j,u\in\PSH(\Omega),  u_j\to u$ in $L^1_{\rm loc}(\Omega)$.
For $\lambda>0$ we assume that
\begin{equation}
\label{eq: main_integral_bound}
             \int_Ue^{-\lambda u}\,d\mu<\infty
             \quad\text{for some }K\subset U\Subset\Omega.
\end{equation}
Then $\lambda<c_{\mu,K}[u]$, and for every
$\lambda<t<c_{\mu,K}[u]$ there exists
$K\subset W\Subset\Omega$ such that $e^{-tu_j},e^{-tu} \in L^1(W,\mu)$ for $j$ large enough and 
\begin{equation}
\label{eq: main_l1_convergence}
       e^{-\lambda u_j}\to e^{-\lambda u}
       \quad\text{in }L^1(W,\mu).
\end{equation}
\end{theorem}

That $\lambda<c_{\mu,K}[u]$ holds is essentially contained in \cite[Corollary B.2]{BBJ21}, which follows from
Guan--Zhou's strong openness \cite{GuZh15}. So the primary goal is to prove the continuity statement of \eqref{eq: main_l1_convergence}. The strategy is to adapt some of Hiep's ideas from \cite{Hi14} to our singular setting. We slice $\Omega$ by the level sets of a holomorphic
function $h$ whose zero set contains
$\operatorname{Sing}X\cup\operatorname{Supp}\Delta$. For a 
nearby regular value $w$, Sard's theorem ensures that $\{h=w\}$ is a
smooth hypersurface, to which the classical Demailly--Koll\'ar
continuity theorem applies. A key step is then to extend holomorphic
functions from $\{h=w\}$ to $\Omega$ using the Ohsawa--Takegoshi
extension theorem of Cao \cite{Cao17}, or that of Zhou--Zhu
\cite{ZhZh18}. Then a Vitali type argument, and basic properties of holomorphic functions on normal spaces prove \eqref{eq: main_l1_convergence}.

\paragraph{Applications to (twisted) alpha invariants.} Let $(X,\Delta)$ be a compact normal K\"ahler klt pair, with $\Delta\geq0$ an effective $\mathbb R$-Weil divisor. Let
$\theta$ be a smooth closed real $(1,1)$-form whose cohomology class
$\{\theta\}$ is big. We fix an adapted measure $d\mu$ for $(X,\Delta)$
and a quasi-plurisubharmonic (qpsh) function $\psi$ on $X$ so that
$e^{-\psi}d\mu$ has unit mass.

The analytic $\psi$-twisted alpha invariant of $(X,\Delta)$ associated with
$\{\theta\}$  is defined by
\[
\alpha_\psi^A(X,\Delta,\{\theta\})
:=
\sup\left\{
\alpha>0:
\sup_{\substack{u\in\operatorname{PSH}_\theta,\  \sup_X u=0}}
\int_X e^{-\alpha u}\, e^{-\psi}d\mu
<\infty
\right\}.
\]

Equivalently, $\alpha_\psi^A(X,\Delta,\{\theta\})$ is the supremum of all $\alpha>0$ for which there
exists a constant $C_\alpha>0$ such that
\[
\int_X e^{-\alpha(u-\sup_X u)}\,e^{-\psi} d\mu
\leq C_\alpha
\]
for every $u\in\operatorname{PSH}_\theta$.
One can also introduce the algebraic/valuative alpha invariant:
\[
\begin{aligned}
\alpha_\psi(X,\Delta,\{\theta\})
&=
\inf_E
\frac{A_{X,\Delta,\psi}(E)}
{\displaystyle\sup_{u\in\operatorname{PSH}_\theta}\nu(u,E)} \\
&=
\inf_E
\frac{A_{X,\Delta,\psi}(E)}
{\tau_{\{\theta\}}(E)},
\end{aligned}
\]
where the infimum ranges over prime divisors $E$ over $X$, represented on a log resolution $\pi:Y\to X$ of $(X,\Delta)$, and $
A_{X,\Delta,\psi}(E):=A_{(X,\Delta)}(E)-\nu(\psi,E)$, with $\nu(\psi,E)$ being the generic Lelong number of $\psi$ along $E$ (with the convention in \cite[(13)]{DZ22}).  Additionally,
\[
\tau_{\{\theta\}}(E)
:=
\sup\left\{
t\geq 0:
\pi^*\{\theta\}-t\{E\}
\text{ is big on }Y
\right\}
\]
is the pseudoeffective threshold of $\{\theta\}$ along $E$. In the case where $\theta\in c_1(L)$ for some big line bundle $L$ and $\psi=0$, this valuative alpha invariant coincides with the global log canonical threshold of $L$ \cite[Theorem C]{BJ20}.

As a first application of Theorem \ref{thm: main_klt} we show that these two invariants coincide:

\begin{theorem}\label{thm: alpha_eq alpha_A} Let $(X,\Delta)$ and $\psi$ be as above. Let $\theta$ be a smooth closed $(1,1)$-form on $X$ such that $\{\theta\}$ is a big class. Then 
$$\alpha_\psi(X,\Delta,\{\theta\}) = \alpha_\psi^A(X,\Delta,\{\theta\}).$$
\end{theorem}

When $X$ is smooth and $\theta\in c_1(L)$ for some big line bundle $L$, this was proved by Berman \cite[Proposition 7.4]{Ber13}, extending earlier results in \cite{CS08,Shi10}. For different approaches to the alpha (and delta) invariant in the case of big line bundles, we refer to \cite{JRT24, JRT25}.

\paragraph{Applications to (twisted) delta invariants.}
Staying with a normal compact K\"ahler klt pair $(X,\Delta)$, with $\Delta \geq 0$, and a big cohomology class $\{\theta\}$, let $\mathcal E^1_\theta$ denote the finite-energy class associated with $\{\theta\}$. We recall the definition of the Monge--Amp\`ere energy \cite{BEGZ10}:
$$I_\theta(u) = \frac{1}{\vol(\{\theta\}) (n+1)}\sum_{j = 0}^n \int_X (u - V_\theta) \theta_u^j  \wedge \theta_{V_\theta}^{n-j}, \ u \in \mathcal{E}^1_\theta.$$ 
Following \cite{Zhang21,Zh21}, the analytic twisted delta invariant of
$(X,\Delta,\psi)$ with respect to $\{\theta\}$ is defined by
\[
\delta_\psi^A(X,\Delta,\{\theta\})
:=
\sup\left\{
\delta>0:
\sup_{u\in\mathcal E^1_\theta}
\int_X e^{-\delta(u-I_\theta(u))}\,e^{-\psi}d\mu
<\infty
\right\}.
\]
Equivalently, $\delta_\psi^A(X,\Delta,\{\theta\})$ is the supremum of
all $\delta>0$ for which there exists a constant $C_\delta>0$ such
that
\[
\int_X e^{-\delta u -\psi}\,d\mu
\leq
C_\delta e^{-\delta I_\theta(u)}
\]
for every $u\in\mathcal E^1_\theta$. So $\delta_\psi^A(X,\Delta,\{\theta\})$ is an optimal Moser--Trudinger exponent, which governs the solvability of KE type equations (see \eqref{eq:  KE_scalar equation_gen} below).

Analogously, inspired by the works \cite{Li17,Fuj19,FO18,BJ20, BJ18}, we define the valuative twisted delta invariant by
\[
\delta_\psi(X,\Delta,\{\theta\})
:=
\inf_E
\frac{A_{X,\Delta,\psi}(E)}
     {S_{\{\theta\}}(E)},
\]
where $E$ ranges over all prime divisors over $X$, and
\[
S_{\{\theta\}}(E)
:=
\frac{1}{\operatorname{vol}(\{\theta\})}
\int_0^{\tau_{\{\theta\}}(E)}
\operatorname{vol}\!\left(
\pi^*\{\theta\}-t\{E\}
\right)\,dt.
\]
Here $\pi:Y\to X$ is a proper bimeromorphic model on which $E$ appears. The
quantity $S_{\{\theta\}}(E)$ is called the expected vanishing order, or
expected Lelong number, of $\{\theta\}$ along $E$.

Since \cite{Zhang21}, the second author initiated the program to prove $\delta=\delta^A$ in various settings. Using quantization techniques, the second author settled the case of smooth polarized manifolds \cite{Zh21}, and later in \cite{DZ24}, relying on pluripotential tools,  we managed to treat general transcendental big classes on singular spaces, but we needed to assume that certain twisted Ding functional is convex. With Theorem \ref{thm: main_klt} in our hands, we can prove the following general result, dropping all the additional restrictions.

\begin{theorem}\label{thm: delta_eq delta_A} Let $(X,\Delta)$ and $\psi$ be as above. Let $\theta$ be a smooth closed $(1,1)$-form on $X$ such that $\{\theta\}$ is a big class. Then
$$\delta_\psi(X,\Delta,\{\theta\}) = \delta^A_\psi(X,\Delta,\{\theta\}).$$
\end{theorem}

The geometric motivation for this equality comes from the solvability
of a natural class of twisted complex Monge--Amp\`ere equations. More
precisely, we seek a divisorial condition ensuring that
\begin{equation}
\label{eq:  KE_scalar equation_gen}
(\theta+\ddc u)^n=e^{-u-\psi}\,d\mu
\end{equation}
admits a solution $u\in\PSH_\theta$ with minimal singularity type. Here $\ddc=\frac{\sqrt{-1}\partial\bar\partial}{2\pi}$.

This equation has a natural K\"ahler--Einstein (KE) interpretation. We set 
$
\theta_u:=\theta+\ddc u.
$
With the convention $\Theta_{h_\Delta}:=-\ddc\log h_\Delta$, we introduce the twisting current  $\eta_\psi:=\eta+\ddc\psi,$ where $\eta:=-\Theta_{h_\Delta}-\theta$. 
With this terminology \eqref{eq:  KE_scalar equation_gen} is equivalent to
\[
\operatorname{Ric}(\theta_u)=
\theta_u+[\Delta]+\eta_\psi \ \textup{ on }   \ X_{\rm reg}.\]
Thus, $\theta_u$ is an $\eta_\psi$-twisted log
KE current on the pair $(X,\Delta)$. 
Even in the smooth setting, it is well known that this equation is not always solvable, and a wide range of algebraic criteria for its solvability have been developed under the umbrella of K-stability \cite{Ti97,Don02}.

An important feature of our framework is that no semipositivity assumption is imposed on the twisting current $\eta_\psi$. When $\eta_\psi\geq0$, results of Berndtsson--P\u{a}un yield  convexity of the corresponding Ding functional (\cite{BePa08,BBEGZ16}). Without this assumption, such convexity is generally unavailable, and the standard variational arguments do not directly apply.

When studying canonical metrics, it is natural to consider some type of continuity method. This is the case here as well, and we consider the following family of equations, for $\lambda > 0$:
\begin{equation}\label{eq:  KE_cont eq}
(\theta + \ddc u_\lambda)^n  = e^{-\lambda u_\lambda  - \psi} d \mu.
\end{equation}
For all sufficiently small $\lambda>0$, the Guan--Zhou openness theorem
\cite{GuZh15} gives $\int_X e^{-\lambda v-\psi}\,d\mu<\infty$
 for every $v\in\mathcal E^1_\theta,$
so the right-hand side is well defined for such $\lambda$.

Following \cite{Ding88}, one defines a functional, whose Euler-Lagrange equation is (a normalized version of) \eqref{eq:  KE_cont eq}. This is $\mathcal D^\lambda: \mathcal{E}^1_\theta \to \RR$, the $\lambda$-Ding functional:
\begin{flalign}\label{def: twisted_Ding}
\mathcal D^\lambda(\varphi): = \mathcal L^\lambda(\varphi) -I_\theta(\varphi):=-\frac{1}{\lambda}\log\int_Xe^{-\lambda\varphi -\psi} \mathrm{d}\mu-I_\theta(\varphi), \ \ \ \ \varphi\in\mathcal{E}^1_\theta.
\end{flalign}

The variational theory of complex Monge--Amp\`ere equations shows that minimizers of $\mathcal D^\lambda$ on $\mathcal E^1_\theta$ solve (the normalized form of) \eqref{eq:  KE_cont eq} \cite{BBGZ13,BBEGZ16}. Moreover, by its definition, $\delta_\psi^A(X,\Delta,{\theta})$ is precisely the threshold governing the boundedness from below of $\mathcal D^\lambda$.

On the other hand,  \cite[Theorem 1.5]{DZ24} gives that the divisorial
invariant $\delta_\psi(X,\Delta,\{\theta\})$ agrees with the geodesic
semistability threshold of $\mathcal D^\lambda$:
\[
\delta_\psi(X,\Delta, \{\theta\}) 
=
\sup\left\{
\lambda>0
\;\middle|\;
\liminf_{t\to\infty} \frac{\mathcal D^\lambda(u_t)}{t}\geq 0, \forall
\text{  sublinear subgeodesic ray }
\{u_t\}\subset \mathcal{E}^1_\theta
\right\}.
\]

For a discussion of (sub)geodesic rays $[0,\infty) \ni t \to u_t \in \mathcal E^1_\theta$,  we refer to \eqref{eq:  subgeod} and \cite[Section 3.3]{DZ24}. We introduce the following notation for the slope of the $\lambda$-Ding energy along such rays, appearing in the above identity:
$$\mathcal D^\lambda \{u_t\}:=\liminf_{t\to\infty} \frac{\mathcal D^\lambda(u_t)}{t}.$$

By the above, Theorem \ref{thm: delta_eq delta_A} will be an immediate consequence of the
following result.
\begin{theorem}
\label{thm:  non_bound_imply_destab_ray} Let $0< \lambda < \delta_\psi(X,\Delta,\{\theta\})$  and assume that $\mathcal{D}^\lambda: \mathcal E^1_\theta \to \Bbb R$ is not bounded from below. Then there exists a non-trivial finite energy geodesic ray $\{v_t\}\in\mathcal{R}^1_\theta$ such that $\mathcal D^{\lambda+ \varepsilon}\{v_t\}<0$ for any  $0< \lambda+\varepsilon  < \delta_\psi(X,\Delta,\{\theta\})$.
\end{theorem}

As a geometric consequence we obtain a divisorial YTD type existence criterion for KE metrics in big cohomology classes, without a semipositivity assumption on the twist:

\begin{corollary} Suppose that $0< \lambda < \delta_\psi(X,\Delta,{\theta})$. Then \eqref{eq:  KE_cont eq} has a solution $u_\lambda \in \PSH_\theta$ with minimal singularity. In particular, if $\delta_\psi(X,\Delta,{\theta})>1$ then 
\eqref{eq:  KE_scalar equation_gen} is solvable, hence a twisted KE current $\theta_{u}$ with minimal singularity exists.
\end{corollary}

For Fano manifolds, when $\{\theta\}=c_1(-K_X)$, YTD type existence theorems have been established using the continuity method and Cheeger--Colding--Tian theory \cite{CDS15,DS16a,Tian15,TiWa20}, the K\"ahler--Ricci flow \cite{CSW}, and variational and non-Archimedean methods \cite{BBJ21}. The second author subsequently gave a more elementary proof based on K\"ahler quantization \cite{Zh21}.

For singular log Fano pairs, Yau--Tian--Donaldson-type existence results were established in \cite{Li20YTD,LTW21a,LTW21b,LXZ22}, building on the pluripotential and variational foundations for singular KE metrics developed in \cite{BBEGZ16}. Related stability questions in the setting where the anticanonical class is merely big were subsequently studied in \cite{DeRe22,DZ24,Xu22}. These works concern projective varieties and therefore have access to techniques from the minimal model program \cite{BCHM10}. By contrast, our results above require no positivity assumptions on $-(K_X+\Delta)$, nor projectivity of $X$.

As pointed out in \cite{DZ24}, the condition
$
\delta_\psi(X,\Delta,{\theta})>1$ is unlikely to be necessary for the existence of twisted KE  metrics in full generality. We nevertheless expect it to be necessary when $\eta_\psi$ is strictly positive, a compelling open problem, having roots in  \cite{Bern15}. We refer to \cite{DZZ} for some recent progress in this direction.

\paragraph{Application to local alpha invariants and normalized volumes.}
Let $X\subset\mathbb C^N$ be an irreducible complex space of pure
dimension $n$ with an isolated log terminal singularity at $p\in X$.
We fix a strongly pseudoconvex neighborhood
$p\in\Omega\Subset X$ such that $K_\Omega$ is $\mathbb Q$-Cartier and
$rK_\Omega$ is trivial for some $r\in\mathbb N$, and let $\mu_p$ be
an adapted probability measure associated with a flat metric $h$ on $K_\Omega$.

In their study of positively curved KE metrics near
isolated log terminal singularities, Guedj and Trusiani \cite{GT23}
introduced the following alpha type invariants
\[
\alpha(X,p)
:=
\sup\left\{
\lambda>0:
\sup_{u\in\mathcal F_1(\Omega)}
\int_\Omega e^{-\lambda u}\,d\mu_p<\infty
\right\}
\]
\[
\widetilde{\alpha}(X,p)
:=
\inf_{u\in\mathcal F_1(\Omega)}
c_{\mu_p,\Omega}[u],
\]
where $ c_{\mu_p,\Omega}[u]
:=
\sup\left\{
c>0:
\int_\Omega e^{-cu}\,d\mu_p<\infty
\right\},$ is the global integrability exponent of $u$ on $\Omega$. See \eqref{eq: F1__local} for the precise definition of the class
$\mathcal F_1(\Omega)\subset\PSH(\Omega)$. They further proved that $\widetilde\alpha(X,p)$ is independent of the choice of
$\Omega$ and that
\[
\widetilde\alpha(X,p)
=
\widehat{\mathrm{vol}}(X,p)^{1/n}.
\]
Here $\widehat{\operatorname{vol}}(X,p)$ denotes the normalized volume of the singularity $(X,p)$, introduced by Chi Li as an algebro geometric extension of the volume-minimization principle of Martelli--Sparks--Yau in Sasaki--Einstein geometry \cite{MSY08,Li18NormalizedVolume}. It is defined by
\[
\widehat{\operatorname{vol}}(X,p)
:=
\inf_{v\in\operatorname{Val}_{X,p}}
A_X(v)^n\operatorname{vol}(v),
\]
where $\operatorname{Val}_{X,p}$ denotes the space of real valuations centered at $p$. For further background and discussion, we refer to \cite{GT23}.

This concept is closely connected with local
K-stability. Blum proved that the infimum in the definition of $\widehat{\operatorname{vol}}(X,p)$ is attained \cite{Blum18}. Li and Xu related minimizing valuations to K-semistable Koll\'ar components and log Fano cones \cite{LiXu20,LiXuHigherRank}. Subsequently, Xu and Zhuang established uniqueness of the minimizer up to rescaling and completed the proof of the stable degeneration conjecture \cite{XuZhuang21,XuZhuang22}.

Guedj--Trusiani identified the absence of a singular Demailly--Koll\'ar continuity
theorem as the main obstruction to proving
the equality
$
\alpha(X,p)=\widetilde\alpha(X,p)
$
in general \cite{GT23}. Our first main theorem provides precisely this missing
piece and yields the following result.

\begin{theorem}
\label{thm: local__alpha_intro}
Let $(X,p)$ be an isolated log terminal singularity as above. Then
\[
\alpha(X,p)
=
\widetilde{\alpha}(X,p)
=
\widehat{\operatorname{vol}}(X,p)^{1/n}.
\]
\end{theorem}

As pointed out in \cite{GT23}, the significance of these invariants stems from their connection with
KE metrics near the singularity $p \in X$. More precisely, for
$\gamma>0$, we consider solutions
\[
u\in\operatorname{PSH}(\Omega)
\cap C^\infty\bigl(\Omega\setminus\{p\}\bigr)
\cap C\bigl(\overline{\Omega}\setminus\{p\}\bigr)
\]
of the Dirichlet problem
\begin{equation}
\label{eq: KE_scalar_local}
\begin{cases}
(\ddc u)^n=\frac{1}{\int_\Omega e^{-\gamma u}\,d\mu_p}e^{-\gamma u}\,d\mu_p & \text{in }\Omega,\\
u=0 & \text{on }\partial\Omega.
\end{cases}
\end{equation}
Geometrically, such a solution determines a singular K\"ahler metric
$\ddc u$
on $\Omega\setminus\{p\}$. Since the adapted measure $d\mu_p$ is
Ricci-flat on $\Omega\setminus\{p\}$, equation
\eqref{eq: KE_scalar_local} is equivalent there to the
KE equation
\[
\operatorname{Ric}(\ddc u)=\gamma \ \ddc u.
\]

Combining Theorem~\ref{thm: local__alpha_intro} with
\cite[Theorems~C and~3.5]{GT23} immediately yields the following
algebro-geometric existence criterion for KE metrics in this setting.

\begin{corollary}
\label{cor: local_KE existence}
Let $(X,p)$ be an isolated log terminal singularity as above. Then the
Dirichlet problem \eqref{eq: KE_scalar_local} admits a solution if 
\[
\gamma
<
\frac{n+1}{n}\,
\widetilde{\alpha}(X,p)=\frac{n+1}{n}\,
\alpha(X,p)=\frac{n+1}{n}\widehat{\operatorname{vol}}(X,p)^{1/n}.
\]
\end{corollary}

In the smooth setting, local Moser--Trudinger inequalities were studied by Guedj--Kolev--Yeganefar and Berman--Berndtsson, with the former also constructing positively curved KE fillings of smoothly bounded strongly pseudoconvex domains \cite{GKY13,BB22}. Guedj--Trusiani subsequently developed a variational theory for isolated log terminal singularities, whose results yield the normalized-volume existence criterion of Corollary~\ref{cor: local_KE existence} under their admissibility assumption \cite{GT23}. Our corollary therefore extends this criterion to arbitrary isolated log terminal singularities.

For affine cone singularities, the above threshold can be described more explicitly. Li proved that K-semistability of the Fano base is equivalent to equivariant minimization of the normalized volume at the canonical valuation, while Li--Liu subsequently established global minimization for cones over KE Fano varieties \cite{Li17,LiLiu19}. Thus, in this setting,  the normalized volume, and the threshold in Corollary~\ref{cor: local_KE existence}, is computed by the canonical valuation.

\paragraph{Acknowledgments.} The first author was partially supported by NSF grant DMS-2405274. The second author is supported by Scientific Research Innovation Capability Support Project for Young Faculty SRICSPYF-ZY2025169 and NSFC grant 12571060.

\paragraph{Use of AI.}
This paper grew out of the authors' efforts to prove 
Theorem~\ref{thm: delta_eq delta_A}. We had been interested in this
problem since our first collaboration \cite{DZ24}, but at that time we were
unable to overcome the difficulty caused by the lack of convexity
of the $\lambda$-Ding functional for $\lambda \geq 1$. During a discussion with ChatGPT 5.6 Sol concerning one of our earlier
approaches, the model suggested applying Demailly--Koll\'ar
continuity together with our \cite[Theorem~1.4]{DZ24}  to the integrability exponents of the Legendre
transforms of converging geodesic segments. This suggestion led us to a proof of Theorem~\ref{thm: delta_eq delta_A} in the smooth K\"ahler setting.

Extending the above argument to klt pairs required a singular version of the Demailly--Koll\'ar continuity theorem. Further discussions with ChatGPT first produced a preliminary global argument in the compact projective setting, and later for local analytic klt pairs. We developed the related ideas, however the idea in the proof of Theorem~\ref{thm: main_klt}, to apply an Ohsawa--Takegoshi type extension theorem along regular holomorphic slices, remains as suggested by the AI.

Apart from AI assisted language editing, the mathematical
arguments were written by us, and we take responsibility for the correctness of the results.

\section{Pluripotential preliminaries on normal spaces}

We collect terminology and a number of preliminary results that will be needed in the proof of our main theorems.

Let $X$ be a reduced normal complex analytic space of pure dimension
$n\geq 1$. We denote its regular and singular locus by
$
X_{\rm reg}$ 
and
$X_{\rm sing}:=X\setminus X_{\rm reg}.
$
Then $X_{\rm reg}$ is a smooth complex manifold of dimension $n$, and 
$X_{\rm sing} \subset X$ is an analytic set of complex codimension at least two.

A function
 $u:X\to \mathbb R\cup\{-\infty\}$
is plurisubharmonic (psh) (notation $u \in \PSH(X)$),  if locally on $X$ it is the restriction
of a psh function defined on an ambient neighborhood in a copy of 
$\mathbb C^N$. 

Recall that a function $v$ is weakly psh if it is
locally bounded from above and its restriction to $X_{\rm reg}$ is psh. Its upper-semicontinuous (usc) extension across $X_{\rm sing}$ is defined by
\[
v^*(p)
:=
\limsup_{\substack{y\to p\\y\in X_{\rm reg}}}v(y).
\]
Since $X$ is normal, it is locally irreducible near all of its points, so the theorem of
Forn{\ae}ss--Narasimhan \cite{FN80} implies that $v$ is weakly
psh if and only if its usc extension $v^*$ is psh on $X$. Hence in our context there is no essential difference between weakly and strongly psh functions.

We say that $w:X\to \mathbb R\cup\{-\infty\}$ is quasi-plurisubharmonic (qpsh) if locally it is the sum of a smooth function and a psh function.

Let us record the following consequence of Guan--Zhou's strong openness \cite{GuZh15}, which is essentially contained in \cite[Corollary B.2]{BBJ21}.

\begin{proposition}
\label{prop: adapted_strong_opn}
Let $(X,\Delta)$ be a normal klt $\Bbb R$-pair and $\mu$ an adapted measure. Let $K\Subset\Omega\subset X$, and
let $u\in\PSH(\Omega)$.  Suppose that $\lambda>0$ and for some
$K\subset U\Subset\Omega$,
\begin{equation}
                       \int_U e^{-\lambda u}\,d\mu<\infty.
\end{equation}
Then there exist $\varepsilon>0$ and
$K\subset U'\Subset U$ such that
\begin{equation}
\label{eq: adapted_strong_opn}
             \int_{U'}e^{-(1+\varepsilon)\lambda u}\,d\mu<\infty.
\end{equation}
In particular, $\lambda<c_{\mu,K}[u].$
\end{proposition}

\begin{proof} By compactness of $K$, it is enough to establish \eqref{eq: adapted_strong_opn} in a neighborhood $W$ of a fixed $x\in K$. After taking a log resolution $\pi:Y\to W$, the measure in \eqref{eq: adapted_density_res} has the local form
\begin{equation}\label{eq: klt_measure}
e^{-\lambda u\circ \pi}d \tilde \mu = e^{P - R} dV,
\end{equation}
where $dV$ is a smooth positive measure and
\[
 P:=\sum_{a_i>0}a_i\log|z_i|^2,
 \qquad
 Q:=\sum_{a_i<0}(-a_i)\log|z_i|^2,
 \qquad
 R:=Q+\lambda u\circ\pi.
\]
Here $P$ has analytic singularities and $R$ is qpsh. As a result, \cite[Corollary B.2]{BBJ21} and its proof are applicable, implying the result.
\end{proof}

We record a local version of the existence of log resolutions, that will readily follow from the results of \cite{Wlo09}.

\begin{proposition}
\label{prop: localresolution}
Let $x$ be a point of a normal complex analytic space $X$, and let
$Z\subsetneq X$ be a proper closed analytic subset.  After shrinking $(X,x)$ inside an embedding into $\Bbb C^N$, there exist strongly pseudoconvex Stein
neighborhoods
\[
  x\in U_0\Subset U_1\Subset U_2 \Subset X
\]
and a proper bimeromorphic holomorphic map
\(
  \pi:V_2\to U_2
\)
with the following properties:
\begin{enumerate}
\item $V_2$ is a smooth K\"ahler manifold, $\pi$  is an isomorphism over
  $U_2\setminus(X_{\Sing}\cup Z)$;
  \item $\operatorname{Exc}(\pi)\cup\pi^{-1}(Z)$ is a simple
  normal-crossing divisor inside $V_2$;
  \item each 
  $V_j := \pi^{-1}( U_j), \ j=0,1$ is  weakly pseudoconvex.
\end{enumerate}
\end{proposition}

\begin{proof} After shrinking $(X,x)$, we can assume that it is irreducible and that there exists a closed embedding into a ball, $X\hookrightarrow\Bbb B_R$ with $x=0$. 
Now carrying out an embedded resolution of $(X,Z)$ for the embedding $X \hookrightarrow \Bbb B_R$, as described in 
\cite[Theorems~2.0.2--2.0.3]{Wlo09}, we get a composition of blowups with smooth centers $q: \tilde {\Bbb B} \to \Bbb B_R$ resolving $(X,Z)$ near $x$. 

Thus, if $\tilde X$ is the strict transform of $X$ under $q$, then we can take $U_j:= \Bbb B_{r_j} \cap X$ and $V_j := q^{-1}(U_j) \cap \tilde X$ for some $0 < r_0 < r_1 < r_2 < R$, yielding the resolution map  
$\pi:= q|_{V_2}: V_2 \to U_2$.

As blowups with smooth centers transform K\"ahler manifolds to K\"ahler manifolds, we get that $V_2$ must be K\"ahler. We obtain that $V_j = \pi^{-1}(U_j)$ is indeed weakly pseudoconvex, with the psh exhaustion function $\rho_j := -\log (r_j^2 - |\pi(\cdot )|^2)$. 
\end{proof}

\section{$L^2$ extension from regular holomorphic slices}

We collect analytic tools needed to extend holomorphic data from regular level sets/slices of a holomorphic function of a normal complex space to the ambient space. These will provide important ingredients for the proof of Theorem \ref{thm: main_klt}.

\subsection{The coarea formula for regular holomorphic fibers}

We  recall the coarea formula for regular holomorphic level sets. Let $T: M \to \Omega \subset \Bbb C$ be a holomorphic map such that $dT$ is non-vanishing, and $(M,\eta)$ is a K\"ahler manifold of dimension $n$. Then for each $w \in \Omega$ the fibers $(T^{-1}(w),\eta)$ are K\"ahler hypersurfaces. 

For any compactly supported nonnegative smooth function $\Psi: M \to \Bbb R$ we have the following formula: 
\begin{flalign}\label{eq: general_coarea_quotient}
 \int_M \Psi\,dV_M = \frac{1}{n!}\int_M \Psi\,\eta^n &= \frac{1}{(n-1)!} \int_\Omega \bigg(\int_{T^{-1}({w})} \Psi \frac{1}{|dT|^2_{\eta}}\eta^{n-1}\bigg) dA(w)\\
 &=\int_\Omega \bigg(\int_{T^{-1}({w})} \Psi \frac{1}{|dT|^2_{\eta}}d V_{T^{-1}(w)}\bigg) dA(w). 
\end{flalign}
This follows from the coarea formula \cite[Theorem~3.2.22]{Fed69}, or more specifically from \cite[Section 2.15, p.~17]{De12}.  We point out the elementary argument. 

Using a partition of unity, it is enough to verify the above formula in local charts of the type $(\xi_1: = z_1,\xi_2=z_2,\ldots, \xi_{n-1} := z_{n-1}, \xi_n := T(z_1,\ldots z_n))$, assuming that $\partial T/\partial z_n \neq 0$.

In such a chart,  $T^{-1}(w) = \{\xi_n = w\}$ and we suppose that $\eta = i \sum_{j,k} \eta_{jk} d \xi_j \wedge \overline{d \xi_k}$. Using the co-factor representation of the inverse matrix, we get that
$$\eta^{n,n} = \frac{\det \eta_{j,k \in \{1,\ldots,{n-1}\}}}{\det \eta_{j,k \in \{1,\ldots,n\}}}.$$
But we also have that $\eta^{n,n} = |d \xi_n|^2_\eta = | d T|^2_\eta$, thus 
$|dT|^2_\eta = \frac{\det \eta_{j,k \in \{1,\ldots,{n-1}\}}}{\det \eta_{j,k \in \{1,\ldots,n\}}}.$

As a result, 
$$  \Psi  \eta^n =  n\frac{\Psi }{|dT|^2_\eta} \eta^{n-1} \wedge \big( i d \xi_n \wedge \overline{d \xi_n} \big).$$
Integrating both sides of this identity on $M$ gives \eqref{eq: general_coarea_quotient}, with the normalization convention $dA(\xi_n):=i d\xi_n\wedge\overline{d\xi_n}$.

\begin{remark}
Using Sard's theorem and the holomorphicity of $T$, we notice that the nowhere-vanishing condition on $dT$ in \eqref{eq: general_coarea_quotient} is superfluous. The formula also holds for any nonnegative Borel measurable function $\Psi:M\to\Bbb R_+$, by an elementary approximation argument.
\end{remark}

\subsection{An $L^2$ extension theorem on weakly pseudoconvex domains}

We record for later use an extension theorem on weakly pseudoconvex domains, which is a codimension one particular case  of
\cite[Theorem~1.1 and Remark~1.3]{ZhZh18} (see also \cite[Theorem 1.3]{Cao17}).

If $(U,\omega)$ is a K\"ahler manifold of dimension $n$ and $Z\subset U$ is a complex hypersurface, then
\[
 dV_{U,\omega}:=\frac{\omega^n}{n!},
 \qquad
 dV_{Z,\omega}:=
 \frac{(\omega|_Z)^{n-1}}{(n-1)!}.
\]
In the following statement, all metrics on canonical tangent/cotangent bundles and their tensor products are
those induced by $\omega$.

\begin{theorem}
\label{thm: zhou_zhu_ext}
Let $(U,\omega)$ be a weakly pseudoconvex K\"ahler manifold of
complex dimension $n$.  Let $L\to U$ be a holomorphic line bundle
equipped with a singular Hermitian metric $h_L$ satisfying
\(
 i\,\Theta_{h_L}(L) =_{loc} -\ddc \log h_L \geq 0
\)
in the sense of currents.

Let $T\in\mathcal O(U)$ and assume that
\(
 Z:=\{T=0\}
\)
is a smooth hypersurface with $d T|_Z$ non-vanishing.  Suppose that, for some
$A\in\mathbb R$,
\(
     \sup_U |T|^2\leq e^A.
\)
If
\(
 f\in H^0\!\left(Z,(K_U\otimes L)|_Z\right)
\)
satisfies
\begin{equation}
\label{eq: zz_slice_integ}
 \int_Z
 \frac{|f|_{\omega,h_L}^2}
      {|dT|_\omega^2}\,
 dV_{Z,\omega}
 <\infty,
\end{equation}
then there exists
\(
 F\in H^0(U,K_U\otimes L),\ F|_Z=f,
\)
such that
\begin{equation}
\label{eq: zhou_zhu_estimate}
 \int_U |F|_{\omega,h_L}^2\,dV_{U,\omega}
 \leq
 2\pi e^A
 \int_Z
 \frac{|f|_{\omega,h_L}^2}
      {|dT|_\omega^2}\,
 dV_{Z,\omega}.
\end{equation}
\end{theorem}

Some comments are in order, before we give the proof.
Under the adjunction isomorphism
\[
 (K_U\otimes L)|_Z
 \simeq
 K_Z\otimes L|_Z\otimes N_{Z/U}^*,
\]
a section $g\in H^0(Z,K_Z\otimes L|_Z)$ determines the section
\[
                    f=dT\wedge g
\]
of $(K_U\otimes L)|_Z$ bijectively.  Pointwise on $Z$,
\[
 |dT\wedge g|_{\omega,h_L}^2
 =
 |dT|_\omega^2\,|g|_{\omega,h_L}^2.
\]
Consequently, Theorem~\ref{thm: zhou_zhu_ext} gives an extension
$F\in H^0(U,K_U\otimes L)$ satisfying $F|_Z = dT \wedge g$ and 
\[
 \int_U |F|_{\omega,h_L}^2\,dV_{U,\omega}
 \leq
 2\pi e^A
 \int_Z |g|_{\omega,h_L}^2\,dV_{Z,\omega}.
\]

\begin{remark}
\label{rem: indpd_ofomega}
    Using the above considerations,  the pointwise norms on both sides of \eqref{eq: zhou_zhu_estimate} become independent of $\omega$. Indeed, for an $L$-valued holomorphic $n$-form $F$ we have  \[ |F|_{\omega,h_L}^{2}\,dV_{U,\omega} = c_ni^{n^2} F\wedge_{h_L}\overline F. \] 
Similarly, under the adjunction isomorphism, write $f$ locally along $Z$ as \[ f=dT\wedge g, \qquad g\in H^0(Z,K_Z\otimes L|_Z). \] Then \[ \frac{|f|_{\omega,h_L}^{2}}{|dT|_{\omega}^{2}}\, dV_{Z,\omega} = |g|_{\omega|_Z,h_L}^{2}\,dV_{Z,\omega} = c_{n-1}i^{(n-1)^2} g\wedge_{h_L}\overline g, \] which is also independent of $\omega$. Here $c_n$ and $c_{n-1}$ are dimensional constants.
\end{remark}

\begin{proof}[Proof of Theorem \ref{thm: zhou_zhu_ext}]
On a component of $U$ which does not meet $Z$, there is no prescribed
restriction, and we set $F=0$.  We may therefore assume that $U$ is
connected and that $Z\neq\varnothing$.

In the notation of \cite[Theorem~1.1]{ZhZh18}, take
\[
 E=U\times\mathbb C,\qquad m=1,\qquad \psi=0,
 \qquad s=e^{-A/2}T,
\]
and equip $E$ with the trivial metric.  Then
\(
 |s|_E^2\leq 1,\ 
 Y:=\{s=0\}=Z.
\)
Since $Z$ is smooth and reduced, $s$ is transverse to the zero
section of $E$.

Because $U$ is connected, $s$ has a zero, and $|s|_E\leq1$, the maximum
principle gives
\(
                       |s|_E<1
\)
everywhere on $U$.  Define
\(
     \alpha:=\frac{1-|s|_E^2}{2}.
\)
This is a continuous strictly positive function on $U$.  Moreover,
the elementary inequality $\log x\leq x-1$ gives, on $U\setminus Z$,
\(
 \log|s|_E^2
 \leq |s|_E^2-1
 =-2\alpha.
\)
Thus condition~(iii) of \cite[Theorem~1.1]{ZhZh18} is satisfied.

The bundle $E$ is flat, and on $U\setminus Z$ the function
$\log|s|_E^2$ is pluriharmonic.  Consequently,
\[
 i\,\Theta_E=0,
 \qquad
 i\,\partial\bar\partial\log|s|_E^2=0
 \quad\text{on }U\setminus Z.
\]
Conditions~(i) and~(ii) of \cite[Theorem~1.1]{ZhZh18} therefore both
reduce to
\(
     i\,\Theta_{h_L}(L)\geq0.
\)
The remaining condition on the singular metric is also automatic:
in a local holomorphic frame write
$h_L=e^{-\varphi}$. This metric is semipositive since  $\varphi$ is psh.

We now choose the denominator function in Zhou--Zhu's theorem. We introduce
\[
                 R(t):=e^{-t},
                 \qquad -\infty<t\leq0.
\]
Then \(R\) belongs to the class used in
\cite[Theorem~1.1]{ZhZh18}, with $C_R:=\int_{-\infty}^0\frac{1}{R(t)} dt=1$ and $e^tR(t)=1$.
That theorem gives an extension $F$ satisfying
\[
 \int_U
 \frac{|F|_{\omega,h_L}^2}
      {e^{\log|s|_E^2}R(\log|s|_E^2)}
 \,dV_{U,\omega}
 \leq
 C_R\frac{(2\pi)^1}{1!}
 \int_Z
 \frac{|f|_{\omega,h_L}^2}
      {|ds|_{\omega,E}^2}
 \,dV_{Z,\omega}.
\]
Note that
\(
 e^{\log|s|_E^2}R(\log|s|_E^2)=1
\)
on $U\setminus Z$, so it extends as $1$ across $Z$. Then we conclude, since $C_R=1$ and
\(
 |ds|_{\omega,E}^2=e^{-A}|dT|_\omega^2.
\)
\end{proof}

\section{Holomorphic estimates on normal spaces}

We will need a number of function theoretic technical results on normal K\"ahler spaces that are likely well known to experts, but we could not easily find in the literature.

\subsection{$L^2$ to $L^\infty$ estimates for klt pairs}

\begin{lemma}
\label{lemma: vol_domination_sup}
Let $(X,\Delta)$ be a normal analytic pair with $\Delta\ge0$, and let $\mu$
be an adapted measure.  We fix nested open sets
\[
  U_0\Subset U_1\Subset U_2\Subset W
\]
contained in one local embedding and in a trivializing patch for the
$\mathbb R$-Cartier data defining $\mu$ (recall \eqref{eq: adapted_measure_def}).   There exist constants $C_{\mathrm{dom}},C_{\mathrm{sup}}<\infty$ such
that
\begin{equation}
  dV_X\le C_{\mathrm{dom}}\,d\mu
  \qquad\text{on }U_2,
  \label{eq: vol_domination}
\end{equation}
and, for every $F\in\mathcal O(U_1)$,
\begin{equation}
  \sup_{U_0}|F|^2
  \le C_{\mathrm{sup}}\int_{U_1}|F|^2\,dV_X.
  \label{eq: l2 sup_normal}
\end{equation}
\end{lemma}
In the above result, $dV_x$ is the local analytic K\"ahler volume of the embedding. The constants $C_{\mathrm{dom}},C_{\mathrm{sup}}$  depend only on $\mu$, the fixed local embedding, the domains $U_0,U_1,U_2$, but not on $F$.

We emphasize that property \eqref{eq: vol_domination} fails in the example described in \cite[Remark~1.3]{Hi14}, thereby accounting for the failure of the conclusion of Theorem~\ref{thm: main_klt} for measures of the form \(e^{\chi-\psi}\omega^n\).

\begin{proof}If $\alpha_I$ are the restrictions of the
coordinate $n$-forms $dz_{i_1}\wedge\cdots\wedge dz_{i_n}$ to $X_{\rm reg}$, then up to a fixed dimensional constant,
\begin{equation}
  dV_X=\sum_{|I|=n}i^{n^2}\alpha_I\wedge\overline{\alpha_I}.
  \label{eq: cauchy_binet_vol}
\end{equation}

We fix $\alpha_I \not \equiv 0$. 
Since $\Delta \geq 0$, $\operatorname{div}(\alpha_I)+\Delta$ is an effective
$\mathbb R$-Cartier divisor.  Since $X$ is normal, after shrinking  $W$ we have that $\operatorname{div}(\alpha_I)+\Delta = \sum_l \rho_l \textup{div}(g_l)$ for some  holomorphic  $g_l$ and $\rho_l > 0$. 

Then by the construction of $d\mu$ (recall \eqref{eq: adapted_measure_def}), we have that (up to a positive smooth multiplying factor)
$$d\mu = \frac{ i^{n^2}\alpha_I \wedge \overline{\alpha_I}} {\Pi_l|g_l|^{2\rho_l}} \ \  \textup{ on }  W.$$
This implies that $(i)^{n^2}\alpha_I \wedge \overline{\alpha_I} \leq C_{\textup{dom}} d\mu$. Coupled with \eqref{eq: cauchy_binet_vol}, \eqref{eq: vol_domination} follows.

We next prove \eqref{eq: l2 sup_normal}.  For every
$x_\alpha\in\overline U_0$, we choose a sufficiently small neighborhood
$Z_\alpha\Subset U_1$ of $x_\alpha$.  By the local parametrization
theorem 
\cite[Chapter~II, \S4.B, Theorem~4.19]{De12}, after choosing a generic linear projection of the fixed ambient
space and shrinking $Z_\alpha$, we obtain a finite proper holomorphic map
\[
 p_\alpha:Z_\alpha\to
 \mathbb B_{R_\alpha}\subset\mathbb C^n.
\]
After translation, we may assume
$p_\alpha(x_\alpha)=0$.  We choose
\[
 0<r_\alpha<s_\alpha<R_\alpha
\]
so that
$
 p_\alpha^{-1}(\overline{\mathbb B}_{s_\alpha})
 \Subset Z_\alpha\Subset U_1.$
By compactness of $\overline U_0$, finitely many of the open sets
$p_\alpha^{-1}(\mathbb B_{r_\alpha})$ cover $\overline U_0$.

We fix $F\in\mathcal O(U_1)$.  For $z \in \Bbb B_{s_\alpha} $ away from the discriminant of $p_\alpha$,
define
\[
 Q_F(z)
 :=
 \sum_{y\in p_\alpha^{-1}(z)}|F(y)|^2,
\]
where the sheets are counted with multiplicity.  On the complement of the
discriminant, the map $p_\alpha$ has finitely many local holomorphic inverse branches,
so $Q_F$ is psh.  Properness and finiteness imply that it is
locally bounded above across the discriminant of $p_\alpha$; hence it extends
plurisubharmonically to $B_{s_\alpha}$.

Let
\[
 \beta_\alpha
 :=
 dp_{\alpha,1}\wedge\cdots\wedge dp_{\alpha,n}.
\]
Since $p_\alpha$ is the restriction of a linear projection of the ambient
space, $\beta_\alpha$ is a fixed linear combination of the coordinate
forms $\alpha_I$.  Therefore, by
\eqref{eq: cauchy_binet_vol},
\[
 p_\alpha^*dV_{\mathbb C^n}
 \le C_\alpha\,dV_X
\]
on $Z_\alpha$. The submean inequality therefore gives 
\begin{align*}
  \sup_{\mathbb B_{r_\alpha}}Q_F
  \le C_\alpha\int_{\mathbb B_{s_\alpha}}Q_F\,dV_{\mathbb C^n}
  = C_\alpha\int_{p_\alpha^{-1}(\mathbb B_{s_\alpha})}
       |F|^2p_\alpha^*dV_{\mathbb C^n}
   \le C_\alpha'\int_{U_1}|F|^2\,dV_X.
\end{align*}
Since $|F|^2\le Q_F\circ p_\alpha$ away from the ramification locus and $F$ is continuous, we get
\[
 \sup_{p_\alpha^{-1}(\mathbb B_{r_\alpha})}|F|^2
 \le C_\alpha'\int_{U_1}|F|^2\,dV_X.
\]
Since a finite number of open sets $p_\alpha^{-1}(\Bbb B_{r_\alpha})$ cover $\overline{U_0}$, we obtain
\eqref{eq: l2 sup_normal}.
\end{proof}

\subsection{Quantitative division along regular holomorphic fibers}

\begin{lemma}
\label{lemma: quant_div}
Let $(X,x)$ be a normal $n$-dimensional analytic germ, let
\[
0\neq h\in\mathcal O_{X,x},\qquad h(x)=0,
\]
and let $U_0\ni x$ be an irreducible representative on which $h$ is
holomorphic.  There exist a neighborhood $x\in U_{-1}\Subset U_0$, a number
$\eta_{\mathrm{div}}>0$, and a constant $C_{\mathrm{div}}<\infty$ such
that
\[
\{h=w\}\cap U_{-1}\neq\varnothing
\qquad\text{for every }|w|<\eta_{\mathrm{div}}.
\]
Additionally, suppose that $0<|w|<\eta_{\mathrm{div}}$ and  $dh$ does not vanish on $\{h=w\}\cap U_0$. If $F\in\mathcal O(U_0)$ satisfies
\[
  F|_{\{h=w\}\cap U_0}=1,
  \qquad \sup_{U_{-1}}|F|\le M,
\]
then
\begin{equation}
  |F(x)-1|\le C_{\mathrm{div}}|w|(M+1).
  \label{eq: quant_div}
\end{equation}

\end{lemma}

\begin{proof}
After possibly shrinking $U_0$ we can choose a proper resolution 
\[
\pi:\widetilde U_0\to U_0.
\]
Let $y\in\pi^{-1}(x)$.  Since
$h$ is a nonzero germ and $\pi$ is bimeromorphic,
$h\circ\pi$
is a nonzero holomorphic germ at $y$.  After a linear change of local
coordinates $(z_1,\ldots,z_n)$ centered at $y=0 \in \Bbb C^n$, the one-variable germ
\[
a(\zeta):=h \circ \pi(0,\ldots,0,\zeta)
\]
is not identically zero.  We choose $r>0$ so small that the closed coordinate
disc
\[
\Gamma:=\{(0,\ldots,0,\zeta):|\zeta|\le r\}
\]
is contained in the coordinate neighborhood and $a$ has no zero on
$|\zeta|=r$.  We introduce
\[
\delta:=\min_{|\zeta|=r}|a(\zeta)|>0,
\qquad
\eta_{\mathrm{div}}:=\frac{\delta}{2},
\]
and choose a neighborhood $x\in U_{-1}\Subset U_0$ so that $\pi(\Gamma)\subset U_{-1}$.  By
Rouch\'e's theorem, for every $|w|<\eta_{\mathrm{div}}$ the function
$a-w$ has a zero in $|\zeta|<r$.  Hence $\{h=w\}\cap U_{-1}\neq\varnothing$.

Now suppose that $0<|w|<\eta_{\mathrm{div}}$ and that $F$ satisfies the
hypotheses of the lemma.  Since $dh \neq 0$ on  $\{h=w\}\cap U_0$, we have $\mathcal I_{\{h = w\}} = (h-w)$. Hence, there exists $G\in\mathcal O(U_0)$ such that 
\begin{equation}
  F-1=(h-w)G.
  \label{eq: division_factor}
\end{equation}
 On $|\zeta|=r$ we have
\[
\begin{aligned}
 |(G\circ\pi)(0,\ldots,0,\zeta)|
 &=\frac{|(F\circ\pi)(0,\ldots,0,\zeta)-1|}{|a(\zeta)-w|}\\
 &\le\frac{M+1}{\delta-|w|}
 \le\frac{2}{\delta}(M+1).
\end{aligned}
\]
Thus, the maximum principle on the disc $|\zeta|\le r$ gives
$
 |G(x)|=|(G\circ\pi)(y)|\le\frac{2}{\delta}(M+1).
$
Since $h(x)=0$, \eqref{eq: division_factor} yields
\[
 |F(x)-1|=|w|\,|G(x)|
 \le\frac{2}{\delta}|w|(M+1),
\]
which proves \eqref{eq: quant_div}, with
$C_{\mathrm{div}}=2/\delta$.
\end{proof}

\subsection{A normal family estimate}

We recall a basic fact about normal families of holomorphic functions on complex spaces.

\begin{lemma}\label{lemma:  normal_fam_est}
Let $\mathcal F\subset\mathcal O(X)$ be a uniformly bounded family of holomorphic functions on a complex space $X$. If $x\in X$ and $|F(x)|\ge 1$ for every $F\in\mathcal F$, then every member of the family is bounded below by $1/2$ on a fixed open neighborhood of $x$.
\end{lemma}

\begin{proof} Suppose no neighborhood of $x$ satisfies the statement.  Then there
are functions $F_k$ in the family and points $x_k\to x$ such that
$|F_k(x_k)|<1/2$.  We choose open neighborhoods $x\in  W \Subset V\Subset X$ and a proper resolution $\pi: Y \to V$.  For all
large $k$, we have $x_k\in W$ and choose $
 y_k\in\pi^{-1}(x_k).$
 
Since $\pi$ is proper, $\pi^{-1}(\overline W)$ is compact. Thus, after perhaps passing to
a subsequence, we may assume that
\[
 y_k\to y_\infty\in\pi^{-1}(x).
\]
The functions
$
 F_k\circ\pi
 $
are honest uniformly bounded holomorphic functions on the manifold $Y$.  By the classical Montel theorem in a
coordinate neighborhood of $y_\infty$, after possibly passing to a further
subsequence, the $F_k \circ \pi$ converge locally uniformly to a holomorphic function
$\widetilde F$.  Since $y_k\to y_\infty$, this gives
\[
 F_k \circ \pi(y_k) - F_k \circ \pi (y_\infty) \to 0.
\]
Since
$ F_k \circ  \pi(y_k)=F_k(x_k), \ 
 F_k \circ \pi (y_\infty)=F_k(x),
$
we get $F_k(x_k)-F_k(x)\to0$, contradicting
$|F_k(x_k)|<1/2$ and $|F_k(x)|\ge 1$.
\end{proof}

\section{Demailly--Koll\'ar continuity for klt spaces}

We now set up the notation for the proof of Theorem \ref{thm: main_klt}.  For the normal space $X$ in the statement of the theorem we introduce
\[
    \Sigma:=X_{\Sing}\cup\Supp\Delta.
\]
We fix $x\in\Sigma$. Near $x$ the set $\Sigma \subset X$ is a proper closed analytic
subset.  Hence its coherent ideal has a nonzero germ
$h\in\mathcal I_{\Sigma,x}$. After shrinking $X$ around $x$,  we can assume that $h \in \mathcal O(X)$ and $h$ vanishes on $\Sigma$.

First we apply Proposition~
\ref{prop: localresolution} with $Z=\Supp\Delta$, giving a resolution $\pi: V_2 \to U_2$ and we fix the
nested resolution domains
\begin{equation}
\label{eq: proof_domains}
x \in U_0 \Subset U_1 \Subset U_2 \Subset X.
\end{equation}

We can assume that $K_X+\Delta$ has a fixed
$\mathbb R$-Cartier presentation with the relevant adapted-measure $d\mu$ trivialized on a neighborhood of $\overline U_1$.  
Using the restriction of the log resolution
$\pi:V_1\to U_1$, we write
\[
 V_b:=\pi^{-1}(U_b)\quad (b=0,1),\qquad T:=h\circ\pi,\qquad
 H_b(w):=\{h=w\}\cap U_b.
\]

By Sard's theorem, 
 for a sufficiently small nonzero regular value $w$, the fiber $H_1(w)$
is contained in $X_{\mathrm{reg}}\setminus\Supp\Delta$.  Writing
$d\mu=\rho\,dV_X$ on $U_1$, we introduce the `conditional measure'
\begin{equation}
\label{eq: conditional_measure_proof}
             d\mu_w:=\frac{\rho}{|dh|^2}\,dV_{H_1(w)}.
\end{equation}
On the right hand side 
$|dh|$ and the hypersurface volume are computed using the Euclidean K\"ahler metric
that defines $dV_X$ on the smooth locus of $X$. 

In the proof of Theorem \ref{thm: main_klt} we will make use of the following $L^2$ extension theorem:

\begin{lemma}
\label{lemma: adapted_slice_ext}
We fix $0<t<t_1$.  Using the neighborhoods 
\eqref{eq: proof_domains}, there exist $\eta_0>0$ and $C_{\rm ext}<\infty$
such that the following holds.  Let $0<|w|<\eta_0$ be a regular value of $h$,
let $v\in\PSH(U_1)$ satisfy $v\le0$, and suppose that
$e^{-t_1v}$ is integrable with respect to $\mu_w$ on a neighborhood in
$H_1(w)$ of
$
    H_1(w)\cap\overline U_0.
$
Then there exists $F\in\mathcal O(U_0)$ such that
\begin{equation}
\label{eq: adapted_ext_conc}
 F=1\quad\text{on }H_0(w),
 \qquad
 \int_{U_0}|F|^2e^{-tv}\,d\mu
 \le C_{\rm ext}\int_{H_0(w)}e^{-tv}\,d\mu_w
 \le C_{\rm ext}\int_{H_0(w)}e^{-t_1v}\,d\mu_w.
\end{equation}
The constant $C_{\rm ext}$ is independent of $v$ and  $w$.
\end{lemma}

\begin{proof}
Write
\[
 K_{V_1}=\pi^*(K_{U_1}+\Delta)+A,
 \qquad D:=\lceil A\rceil = \sum_j \lceil a_j \rceil E_j ,
 \qquad B:=D-A = \sum_j (\lceil a_j \rceil - a_j) E_j.
\]
The klt inequalities give $D\ge0$ and $0\le \lceil a_j \rceil - a_j<1$.  On the strict
transform of a component of $\Delta$, with coefficient $d_i\in[0,1)$, the
coefficient of $A$ is $-d_i$, with ceiling zero.  Thus $D$ is
$\pi$-exceptional.  Normality of $X$ and the Riemann extension theorem give
\begin{equation}
\label{eq: exceptional_push_short}
    \pi_*\mathcal O_{V_1}(D)=\mathcal O_{U_1}.
\end{equation}

We introduce $P_\Delta:=K_{U_1}+\Delta$.  By the choice of $U_1$, this divisor has a fixed
$\mathbb R$-Cartier presentation
\[
    P_\Delta=\sum_{\alpha=1}^N c_\alpha P_\alpha
\]
in which every Cartier divisor $P_\alpha$ induces a trivial line bundle $\mathcal O(P_\alpha)$ on $U_1$.  We choose a nonvanishing holomorphic frame of each $\mathcal O_{U_1}(P_\alpha)$ and
give it the flat metric $h_\alpha$.  Their product with exponents
$c_\alpha$ is a formal flat smooth metric $h_{P_\Delta} : = \Pi_\alpha h_\alpha^{c_\alpha }$ on the formal
$\mathbb R$-line bundle $\mathcal O_{U_1}(P_\Delta)$.

We introduce the line bundle
    \[\mathcal M:=\mathcal O_{V_1}(D-K_{V_1}).\]
The identity of $\mathbb R$-Cartier divisors
$D-K_{V_1}=B-\pi^*P_\Delta$
proves that $
 \mathcal O_{V_2}(B)\otimes\pi^*\mathcal O_{U_1}(-P_\Delta)$
is the genuine line bundle $\mathcal M$.  Equip $\mathcal O_{V_1}(B)$ with
its formal divisor metric $h_B = \Pi_j h^{\lceil a_j \rceil - a_j}_{E_j}$, where $i\,\Theta_{h_{E_j}}=[E_j]\ge0$. We  introduce
\[
 h_{\mathcal M}:=h_{B}
\otimes\pi^*h_{P_\Delta}^{-1},
\]
a hermitian metric on $\mathcal M$.  Thus $h_\mathcal M$ is a singular
Hermitian metric on the genuine line bundle $\mathcal M$.  Since $h_{P_\Delta}$ is flat,
\[
    i\,\Theta_{h_\mathcal M}(\mathcal M)=[B]\ge0.
\]

The weighted metric  $h_\mathcal M e^{-tv\circ\pi}$ has positive curvature current since $v$ is psh on $U_1$.  
                 
Using the identification, 
$K_{V_1}\otimes\mathcal M\simeq\mathcal O_{V_1}(D)$, we view the canonical section  $s_D$ of $D$ as an $\mathcal M$-valued holomorphic $n$-form.  

Now we use the K\"ahler form $\omega$ on $V_2$ furnished by Proposition~
\ref{prop: localresolution}, together with its induced canonical
bundle metric and volume forms.  

Let $d\widetilde\mu$ be the pullback of
$d\mu$ on the isomorphism locus of $\pi$, extended across the resolution
divisor by the formula of 
\eqref{eq: adapted_density_res}. 

For a nonzero regular value
$w$, set
$d\widetilde\mu_w:=\pi^*d\mu_w$ on $\{h \circ \pi =w\}$. This measure is well defined because
$\pi$ is an isomorphism over the regular  slice $\{h \circ \pi =w\}$. 

The discrepancy formula gives, on
the fixed set $V_1$, the following identities. There exists a smooth
positive function $q\in C^\infty(V_1)$, independent of $v$ and of the
small regular value $w$, such that
\begin{align}
 |s_D|^2_{\omega, h_{\mathcal M} e^{- t v \circ \pi }}  dV_{V_2, \omega}
 &=q e^{-tv\circ\pi}\,d\widetilde\mu,
 \label{eq: bulk_norm_short}\\
 \frac{|s_D|^2_{\omega,h_{\mathcal M} e^{-t v\circ \pi}}}{|d (h \circ \pi)|_{\omega}^2}
       dV_{\{h \circ \pi =w\},\omega}
 &=(q e^{-tv\circ\pi})|_{\{ h \circ \pi =w\}}\,d\widetilde\mu_w,
 \qquad \pi_*\widetilde\mu_w=\mu_w.
 \label{eq: residue_norm_short}
\end{align}
Indeed, due to Remark \ref{rem: indpd_ofomega} both measures on the left are intrinsic (independent of $\omega$). Additionally, in snc coordinates the singular factor of $h_\mathcal M$ is
$\prod_i|z_i|^{-2(\lceil a_i \rceil - a_i)}$. Thus, the only singular factor in the above formulas is
\[
       |z_i|^{-2(\lceil a_i \rceil - a_i)}|z_i|^{2\lceil a_i \rceil}=|z_i|^{2a_i}.
\]
All remaining factors are fixed smooth positive functions. 
Since $\overline V_0\Subset V_1$, the function $q$ has uniform positive
upper and lower bounds on $\overline V_0$.

We apply Theorem~\ref{thm: zhou_zhu_ext} to the defining function
$h \circ \pi -w$, with
\[
 U=V_0,\qquad L=\mathcal M,
 \qquad f=s_D|_{\{h \circ \pi =w\}\cap V_0}.
\]
Every component of the resolution divisor maps into $\Sigma$, where
$h=0$; hence $h \circ \pi$ vanishes along it and a nonzero fiber avoids
$\Supp(D+B)$.  We invoke Lemma~
\ref{lemma: quant_div}, and possibly decrease $\eta_0$ so that 
$\eta_0<\eta_{\mathrm{div}}$,  ensuring $H_0(w)\neq\varnothing$ for
all $0<|w|<\eta_0$.

Due to \eqref{eq: residue_norm_short}, the finiteness condition 
\eqref{eq: zz_slice_integ} holds for the section $f$.  We fix $A_0 >0$ such that
\[
    \sup_{|w|<\eta_0}\sup_{V_0}|h \circ \pi-w|^2\le e^{A_0}.
\]
Putting everything together,  Theorem~\ref{thm: zhou_zhu_ext} gives
\[
 S\in H^0(V_0,K_{V_1}\otimes\mathcal M)=H^0(V_0,\mathcal O_{V_2}(D)),
 \qquad S|_{\{T=w\}\cap V_0}=s_D,
\]
with the following $L^2$ bound, where the constants are independent of $v$ and $|w| < \eta_0$:
$$
\int_{V_0} |S|_{\omega,h_\mathcal M e^{-t v \circ \pi}}^2\,dV_{V_2,\omega}
 \leq
 2\pi e^{A_0}
 \int_{H_0(w)}
 \frac{|s_D|_{\omega,h_\mathcal M e^{-t v \circ \pi}}^2}
      {|dT|_\omega^2}\,
 dV_{\{h \circ \pi =w\},\omega}.$$

Using
\eqref{eq: exceptional_push_short}, we have a bijection 
between $H^0(V_0,\mathcal O_{V_2}(D))$ and $H^0(U_0,\mathcal O_{U_0})$, using the
natural map $F\mapsto(F\circ\pi)s_D$.  

As a result, 
$S=(F\circ\pi)s_D$ for a unique $F\in\mathcal O(U_0)$.  The condition $S|_{\{T=w\}\cap V_0}=s_D$  gives $F=1$ on $H_0(w)$. Thus, the formulas 
\eqref{eq: bulk_norm_short}, \eqref{eq: residue_norm_short} give the first estimate of \eqref{eq: adapted_ext_conc}. Since $v\le0$ and $t<t_1$ the second estimate follows immediately.
\end{proof}

In order to employ a Vitali type argument  in the proof of Theorem \ref{thm: main_klt}, we will need a uniform integrability estimate that we now supply:

\begin{proposition}\label{prop: local_higher_moment}
Let $G_0\Subset G\subset X$ be nested open sets and $x \in  G_0$. Let
$u_j,u\in\PSH(G)$, and assume $u_j\to u$ in
$L^1_{\rm loc}(G)$.  If
\begin{equation}
\label{eq: t0 integrability_short}
                  \int_{G_0}e^{-t_0u}\,d\mu<\infty
\end{equation}
for some $t_0>0$, then for every $0<t<t_0$ there is a neighborhood
$x \in W_x\Subset G_0$ of $x$, an index $j_x$, and $C_x<\infty$ such that
\begin{equation}
\label{eq: higher_moment_short}
        \sup_{j\ge j_x}\int_{W_x}e^{-tu_j}\,d\mu\le C_x.
\end{equation}
\end{proposition}

\begin{proof}
Since $u_j\to u$ in $L^1_{\rm loc}(G)$, Hartogs' lemma gives a uniform
upper bound for $u_j$ on $\overline{G_0}$.  Since $u$ is also bounded above
there, after subtracting a common constant we may assume
\begin{equation}
\label{eq: negative_norm}
         u_j\le0,\qquad u\le0
        \quad\text{on }G_0.
\end{equation}

If $x\notin\Sigma$, the assertion follows from the classical
Demailly--Koll\'ar theorem \cite[Main Theorem~0.2(2)]{DK01}.  We assume
$x\in\Sigma$, and fix
\[
         t<t_1<t_0.
\]
We choose the local data from \eqref{eq: proof_domains} so that
\[
   x\in U_0\Subset U_1\Subset U_2\Subset G_0.
\]
Applying Lemma~\ref{lemma: quant_div} to $h$ on $U_0$, we fix
\[
 x\in U_{-1}\Subset U_0,\qquad
 \eta_{\rm div}>0,\qquad C_{\rm div}<\infty.
\]
We also fix the constants of Lemma~\ref{lemma: adapted_slice_ext} for $U_0, U_1$: 
$$\eta_0,C_{\rm ext}.$$ 
Below, $C>0$ denotes a constant
depending only on this fixed local data and on the exponents $t,t_1,t_0$, and may
change from line to line.

Suppose, toward a contradiction, that \eqref{eq: higher_moment_short}
fails.  We choose a decreasing neighborhood basis $W_k\downarrow\{x\}$ with
$W_k\Subset U_{-1}$.  Passing to a subsequence, still denoted $u_k$, we
may arrange that
\begin{equation}
\label{eq: bad subsequence_short}
 \int_{U_1}|u_k-u|\,dV_X<2^{-k},
 \qquad
 \int_{W_k}e^{-tu_k}\,d\mu>k.
\end{equation}

By Fubini's theorem and a standard measure theoretic argument, it follows that, for almost every regular value $w$ satisfying $0<|w|<\eta_0$,
\begin{equation}
\label{eq: slice_l1_short}
u_k|{H_1(w)}\longrightarrow u|{H_1(w)}
\quad\text{in }L^1\bigl(H_1(w)\bigr).
\end{equation}

Using $d\mu=\rho\,dV_X$ and
\[
    d\mu_w=\frac{\rho}{|dh|^2}\,dV_{H_1(w)},
\]
the coarea formula \eqref{eq: general_coarea_quotient}   gives
\begin{equation}
\label{eq: weighted_coarea}
 \int_{|w|<\eta}
\left(\int_{H_1(w)}e^{-t_0u}\,d\mu_w\right)dA(w)
 =
 \int_{\{|h|<\eta\}\cap U_1}e^{-t_0u}\,d\mu.
\end{equation}
We note that the right-hand side tends to zero as $\eta\searrow 0$. Indeed,
$e^{-t_0u}\mu$ is finite on $U_1$ by
\eqref{eq: t0 integrability_short}, and it puts no mass on $\{h=0\}$.

We apply Lemma~\ref{lemma: vol_domination_sup} to
\[
        U_{-1}\Subset U_0\Subset U_1\Subset U_2.
\]
Using \eqref{eq: negative_norm}, for 
$F\in\mathcal O(U_0)$ we obtain that
\begin{equation}
\label{eq: combined sup}
\sup_{U_{-1}}|F|^2
\le C\int_{U_0}|F|^2\,d\mu
\le C\int_{U_0}|F|^2e^{-tu_k}\,d\mu.
\end{equation}

We fix $\varepsilon>0$, to be specified later.  We choose
$ 0<\eta := \eta_\varepsilon<\min\{\eta_0,\eta_{\rm div}\}$
sufficiently small so that
\begin{equation}
\label{eq: eta_choice_short}
 \int_{\{|h|<\eta\}\cap U_1}e^{-t_0u}\,d\mu\le\varepsilon^2.
\end{equation}
Averaging \eqref{eq: weighted_coarea} over
$\eta/2<|w|<\eta$, there exists  $w: = w_\eta$ inside the annulus $A(\frac{\eta}{2},\eta)$ such that 
\eqref{eq: slice_l1_short} holds, $w$ is regular for $h$, 
and
\begin{equation}
\label{eq: good_slice_short}
\int_{H_1(w)}e^{-t_0u}\,d\mu_w
       \le C\frac{\varepsilon^2}{|w|^2}.
\end{equation}
Since $\eta<\eta_{\rm div}$, Lemma~\ref{lemma: quant_div} also
gives $H_0(w)\neq\varnothing$.

For such values $w$, the measure $\mu_w$ is boundedly
equivalent to a smooth positive hypersurface volume near
$H_1(w)\cap\overline U_0$.  By \eqref{eq: slice_l1_short} and the
classical Demailly--Koll\'ar theorem, applied on a finite cover at the
exponent $t_1<t_0$,
\[
 e^{-t_1u_k}\to e^{-t_1u}
 \quad\text{in }L^1(\mu_w)
\]
on a neighborhood of $H_1(w)\cap\overline U_0$.  Since
$u\le0$ and $t_1<t_0$, we have $e^{-t_1u}\le e^{-t_0u}$; hence, for all
large $k$,
\begin{equation}
\label{eq: slice_bound_short}
 \int_{H_0(w)}e^{-t_1u_k}\,d\mu_w
 \le2\int_{H_1(w)}e^{-t_0 u}d\mu_w
 \le C\frac{\varepsilon^2}{|w|^2}.
\end{equation}

Lemma~\ref{lemma: adapted_slice_ext}, applied with $v=u_k$, now gives
$F_k\in\mathcal O(U_0)$ such that
\begin{equation}
\label{eq: extension_bound_short}
 F_k=1\quad\text{on }H_0(w),
 \qquad
 \int_{U_0}|F_k|^2e^{-tu_k}\,d\mu
 \le C\frac{\varepsilon^2}{|w|^2}.
\end{equation}
Combining \eqref{eq: combined sup} and
\eqref{eq: extension_bound_short}, we get
\begin{equation}
\label{eq: sup_bound_short}
             \sup_{U_{-1}}|F_k|
             \le C\frac{\varepsilon}{|w|}.
\end{equation}

The constant $C$ on the right hand side does not depend on $F_k, u_k, w,\eta$ or $\varepsilon$. Knowing this, we  now shrink  $\varepsilon>0$ and 
$\eta>0$ to ensure
\[
 C_{\rm div} C\varepsilon<\frac18,
 \qquad
 C_{\rm div}\eta<\frac18.
\]
Additionally, we now fix a regular value $w \in A(\frac{\eta}{2},\eta)$ satisfying all the above conclusions.  Lemma~\ref{lemma: quant_div} then yields
\[
\begin{aligned}
 |F_k(x)-1|
 &\le C_{\rm div}|w|
       \left(C\frac{\varepsilon}{|w|}+1\right)\le C\varepsilon+C_{\rm div}|w|
 <\frac14.
\end{aligned}
\]
Thus
\[
   |F_k(x)|\ge\frac34.
\]

Since $w$ is fixed, \eqref{eq: sup_bound_short} shows that $(F_k)$ is
uniformly bounded on $U_{-1}$.  Then
$|\frac{4}{3}F_k(x)|\ge1$, so Lemma~\ref{lemma:  normal_fam_est} gives a fixed
neighborhood $x\in W\Subset U_{-1}$ such that $|F_k|\ge\frac38$ on $W$.  Therefore
\[
 \int_W e^{-tu_k}\,d\mu
 \le\frac{64}{9}\int_{U_0}|F_k|^2e^{-tu_k}\,d\mu
 \le C_w,
\]
where $C_w<\infty$ is independent of $k$.  For all sufficiently large
$k$, we have $W_k\subset W$ and $k>C_w$, contradicting
\eqref{eq: bad subsequence_short}.  This proves the proposition.
\end{proof}

\begin{proof}[Proof of Theorem~\ref{thm: main_klt}]
Proposition~\ref{prop: adapted_strong_opn}, applied to
\eqref{eq: main_integral_bound}, first gives
$\lambda<c_{\mu,K}[u]$.
We fix $\lambda<t<s < c_{\mu,K}[u]$ and
$K\subset U\Subset\Omega$ such that
\[
                         \int_Ue^{-su}\,d\mu<\infty,
\]
 Since $\mu(U)<\infty$, H\"older's inequality also
gives $e^{-tu},e^{-\lambda u}\in L^1(U,\mu)$.
Proposition~\ref{prop: local_higher_moment}, applied
at finitely many points giving a covering of $K$, gives
$K\subset W\Subset U$, $j_0$, and $C<\infty$ such that
\[
 e^{-tu}\in L^1(W,\mu),
 \qquad
 \sup_{j\ge j_0}\int_We^{-tu_j}\,d\mu\le C.
\]

It remains to prove the $L^1$-convergence $e^{-\lambda u_j}\to e^{-\lambda u}$.  On a finite collection of
local embeddings covering $\overline W$, we patch the Euclidean K\"ahler
volumes by a partition of unity, to obtain one finite reference measure
$dV$ on $W$.

Since  $d\mu=g\,dV$ with $g\in L^1(dV)$, we get that
$e^{-\lambda u_j}\to e^{-\lambda u}$ in $\mu$-measure as well.  This means that for any $k \in \Bbb N$ 

$$\mu(R_{k,j}):= \mu \{x \in W, \ |e^{-\lambda u_j(x)}- e^{-\lambda u(x)}| > 1/k\} \to 0, \ \ \ j \to \infty.$$

For
$q:=t/\lambda>1$ we have
\[
 \sup_{j\ge j_0}\int_W(e^{-\lambda u_j})^q\,d\mu
 =\sup_{j\ge j_0}\int_We^{-tu_j}\,d\mu\le C  \ \ \textup{ and } \ \ \int_W (e^{-\lambda u})^q d\mu \leq C.
\]
Thus, using H\"older's inequality with exponents $(q,q/(q-1))$, the Vitali type argument gives that 
\begin{flalign*}
\int_{W} |e^{-\lambda u_j(x)}- e^{-\lambda u(x)}| d\mu &= \bigg(\int_{W \setminus  R_{k,j}} + \int_{R_{k,j}}\bigg)|e^{-\lambda u_j(x)}- e^{-\lambda u(x)}| d\mu\\
&\leq C \mu(R_{k,j})^{\frac{q-1}{q}} + \frac{1}{k} \mu(W)\to \frac{1}{k} \mu(W), \ \ j \to \infty.
\end{flalign*}
We conclude that  $e^{-\lambda u_j}\to e^{-\lambda u}$ in $L^1(W,\mu)$.
\end{proof}

\section{Twisted alpha invariants}

Let $X$ be a compact normal K\"ahler space and let
$\Delta\geq0$ be an effective $\mathbb R$-Weil divisor such that
$K_X+\Delta$ is $\mathbb R$-Cartier and $(X,\Delta)$ is klt. Let $\psi$ be a qpsh function on $X$ such that $e^{-\psi} d\mu$ has unit mass, where $d\mu$ is an adapted measure for $(X,\Delta)$.

\begin{theorem}[=Theorem \ref{thm: alpha_eq alpha_A}]\label{thm: alpha_eq alpha_A_later} Let $(X,\Delta)$ and $\psi$ be as above. Let $\theta$ be a smooth closed $(1,1)$-form on $X$ such that $\{\theta\}$ is a big class. Then
$$\alpha_\psi(X,\Delta,\{\theta\}) = \alpha_\psi^A(X,\Delta,\{\theta\}).$$
\end{theorem}

\begin{proof} Suppose that $\lambda < \alpha^A_\psi(X,\Delta,\{\theta\})$. This implies that
$$\int_X e^{-\lambda u - \psi} d\mu < C_\lambda $$
for $u \in \PSH_\theta$ with $\sup_X u =0.$

If $E \subset Y$ is a prime divisor, where $\pi: Y \to X$ is a log resolution of $(X,\Delta)$, then the (log) valuative criterion for integrability \cite[Appendix B]{BBJ21} gives
$$\lambda < \frac{A_{X,\Delta,\psi}(E)}{\nu(u,E)}.$$
Alternatively, one can also use \cite[Theorem 2.3]{DZ24} for the potentials $u \circ \pi,\psi\circ \pi$ and $\pi^* d\mu$ on $Y$, to conclude the above estimate. Taking infimum over $u$ and then $E$, we obtain that $\lambda \leq \alpha_\psi(X,\Delta,\{\theta\}).$

Conversely, suppose that $\lambda < \alpha_\psi(X,\Delta,\{\theta\})$. 
Let $\lambda' \in (\lambda, \alpha_\psi(X,\Delta,\{\theta\}))$. By the same log valuative criterion, we obtain that  $\int_X e^{-\lambda 'u -\psi} d\mu < \infty$ for all $u \in \PSH_\theta$.
We claim that there exists $C_\lambda$ such that 
\begin{equation}\label{eq:  alpha_ineq}
\int_X e^{-\lambda u - \psi} d\mu < C_\lambda 
\end{equation}
for $u \in \PSH_\theta$ with $\sup_X u =0.$ 
Indeed, if this were not the case, then there would exist $u_j \in \PSH_\theta$ with  $\sup_X u_j =0$, such that $\int_X e^{-\lambda u_j - \psi} d\mu> j$.

By Hartogs' theorem, after possibly passing to a subsequence, we can assume that $u_j \to u \in \PSH_\theta$. Applying Theorem~\ref{thm: main_klt} to the quasi-psh functions $u_j+\lambda^{-1}\psi$ and $u+\lambda^{-1}\psi$, we obtain $\int_X e^{-\lambda u_j - \psi} d\mu \to \int_X e^{-\lambda u - \psi} d\mu < \infty$, a contradiction. Thus \eqref{eq:  alpha_ineq} must hold and  $\lambda \leq \alpha^A_\psi(X,\Delta,\{\theta\}).$
\end{proof}

\section{Twisted delta invariants}

\subsection{Geodesic preliminaries}

We recall some facts about the Ross--Witt Nystr\"om correspondence \cite{RWN14} for weak subgeodesic rays. For more details we refer to \cite[Section 3]{DZ24}.

We recall the Legendre transform, which establishes the duality between various types
of maximal test curves and geodesic rays. Given a convex function
$
f\colon [0,\infty)\to \mathbb{R},
$
its Legendre transform is defined by
\begin{equation}
\label{eq:  Legendre_def}
\widehat{f}(\tau)
:=
\inf_{t\geq 0}\bigl(f(t)-t\tau\bigr)
=
\inf_{t>0}\bigl(f(t)-t\tau\bigr),
\qquad \tau\in\mathbb{R}.
\end{equation}
The inverse Legendre transform of a decreasing concave function
$
g\colon\mathbb{R}\to\mathbb{R}\cup\{-\infty\}
$
is defined by
\[
\check{g}(t)
:=
\sup_{\tau\in\mathbb{R}}\bigl(g(\tau)+t\tau\bigr),
\qquad t\geq 0.
\]
Careful readers will notice a sign difference between our convention for the Legendre transform and
the one commonly used in the convex analysis literature \cite{Rock70}. However, our convention
is better suited to the discussion of subgeodesic rays.

Let us start with a technical lemma:

\begin{lemma}
Let $f_j,f:[0,\infty) \to (-\infty,\infty]$ be lsc convex functions such that   $\lim_{t \to \infty}\frac{f(t)}{t} =0$, $f_j(0), f(0)\in\mathbb R$ and $f_j \to f$ uniformly on compact subsets of $[0,\infty)$. Then for any $\tau<0$ we have that $\hat f_j(\tau) \to \hat f(\tau).$
\end{lemma}

\begin{proof} Subtracting $f_j(0)$ from $f_j$  and $f(0)$ from $f$, we may assume without loss of generality that $f_j(0)=f(0)=0$.

Since $\lim_{t \to \infty}\frac{f(t)}{t} =0$, for any $\varepsilon \in (0,-\frac{\tau}{2})$ fixed we have $f(t) \geq -\varepsilon t$ for $t \geq t_\varepsilon$.

Using convexity, $f_j(0)=f(0)=0$, and uniform convergence on compacts, for $j \geq j_\varepsilon$ we must have $f_j(t) \geq -2 \varepsilon t, \ t \geq t_\varepsilon$. Thus, the infimum of $t \to f_j(t) - \tau t$ on $[0,\infty)$ must be obtained on $[0,t_\varepsilon]$, since $f_j(t) - \tau t \geq (-\tau-2 \varepsilon) t, \ t \geq t_\varepsilon$.

But using uniform convergence on $[0,t_\varepsilon]$ we obtain that:
$$\hat f_j(\tau) = \inf_{t \in [0,\infty)}(f_j(t) - \tau t)=\inf_{t \in [0,t_\varepsilon]}(f_j(t) - \tau t) \to \inf_{t \in [0,t_\varepsilon]}(f(t) - \tau t) =\inf_{t \in [0,\infty)}(f(t) - \tau t)=\hat f(\tau).$$  
\end{proof}

Let $(X,\omega)$ be a compact normal K\"ahler space with a big cohomology class $\{\theta\}$. By $\PSH_\theta$ and $\mathcal E^1_\theta$ we denote the space of $\theta$-psh functions and finite energy $\theta$-psh functions, as introduced in \cite{BEGZ10}. Recall that $V_\theta \in \PSH_\theta$ is the $\theta$-psh function with minimal singularity type:
$$V_\theta := \sup \{v \in \PSH_\theta, \ v\leq 0\}.$$

Recalling the definition from \cite{DZ24}, a \emph{sublinear subgeodesic ray} is a subgeodesic
\begin{equation}\label{eq:  subgeod}
(0,\infty)\ni t\mapsto u_t\in\operatorname{PSH}_\theta
\qquad\text{(notation: $\{u_t\}$)}
\end{equation}
satisfying  $u_t\xrightarrow{L^1}u_0:=V_\theta$ as $t\to 0$, and there exists a constant \(C\in\mathbb{R}\) such that
$
u_t(x)\leq Ct, \  t\geq 0, \ x\in X.
$
In addition, \(\{u_t\}\) is said to be \emph{of finite energy} if
$
u_t\in\mathcal{E}^1_\theta, t\geq 0.
$

We refer to \cite{DDNL3} for the connections between finite energy geodesics and the $d_1$ metric geometry of $\mathcal E^1_\theta$.

\medskip

A \emph{psh geodesic ray} is a sublinear subgeodesic ray that additionally
satisfies a maximality property: for any \(0<a<b\), the
subgeodesic
\[
(0,1)\ni t\mapsto
v_t^{a,b}:=u_{a(1-t)+bt}\in\operatorname{PSH}_\theta
\]
can be recovered as
\begin{equation}
\label{eq: geodesic_envelope}
v_t^{a,b}
=
\sup_{h\in\mathcal{S}}h_t,
\qquad t\in[0,1],
\end{equation}
where \(\mathcal{S}\) is the set of subgeodesic segments
$
(0,1)\ni t\mapsto h_t\in\operatorname{PSH}_\theta
$
satisfying
$
\lim_{t\searrow 0}h_t\leq u_a, \ 
\lim_{t\nearrow 1}h_t\leq u_b.$
The space of finite energy psh geodesic rays will be denoted by
$
\mathcal{R}^1_\theta.
$

Adapting \eqref{eq:  Legendre_def}, the Legendre transform of a subgeodesic is defined as follows:
\begin{equation}\label{eq:  Leg_transf_seg_ray}
    \begin{aligned}
\hat u_\tau(x) := \inf_{t}(u_t(x) - t\tau), \ \ \  \tau \in \mathbb R, \ x \in X.
\end{aligned}
\end{equation}
By Kiselman's minimum principle, these potentials $\hat u_\tau$ are either $-\infty$ identically or they are elements of $\PSH_\theta$. \medskip

Later we will need the following approximation result, about subgeodesic segments converging to subgeodesic rays:

\begin{proposition}\label{prop:  L^1_conv_Lagrange_dual} Suppose that $[0,l_j] \ni t \to u^j_t \in \mathcal{E}^1_\theta$ are subgeodesics with $\sup_X u^j_t =0$ and $l_j \nearrow \infty$. Let $[0,\infty) \ni t \to u_t \in \mathcal{E}^1_\theta$ be a subgeodesic ray with $u^j(t,x) := u^j_t(x) \to_{L^1} u(t,x) : = u_t(x)$ locally on $(0,\infty) \times X$. Additionally, we assume that $u_t\to V_\theta$ in $L^1(X)$ as $t\to0$ and that $u_0^j=V_\theta$ for every $j$.

Then, up to further taking subsequences, for any $\tau < 0$ we have that $\hat u^j_\tau \to_{L^1(X)} \hat u_\tau$.
\end{proposition}

\begin{proof} Let us fix $\tau<0$. Extend $t\mapsto u_t^j(x)$ by $\infty$ for $t>l_j$. Due to the convergence conditions, \cite[Lemma 3.2]{DZ24}, and $t$-convexity, there exists $E \subset X$ of Lebesgue measure zero such that for any $x \in X \setminus E$  the functions $[0,l_j] \ni t \to u^j_t(x) \in \Bbb R$ and $[0,\infty) \ni t \to u_t(x) \in \Bbb R$ are finite valued continuous convex functions, $u^j_t(x) \to u_t(x)$ uniformly on compacts of $[0,\infty)$, and  $u^j_0(x) = u_0(x) = V_\theta(x)$.

Additionally, since $\sup_X u_t =0$, by Hartogs' lemma, we get that $u_t \to u_\infty \in \PSH_\theta$ as $t\to\infty$, hence we can also assume that $u_t(x)/t \to 0$.

As a result, the previous lemma is applicable, and we get that $\hat u^j_\tau(x) \to \hat u_\tau(x)$, $x \in X \setminus E$. Thus $\hat u^j_\tau \to \hat u_\tau$ pointwise a.e. on $X$. Since it is well known that $\sup_X \hat u^j_\tau = \sup_X \hat u_\tau$, we get from Hartogs' lemma that $\hat u^j_\tau \to_{L^1} \hat u_\tau$, as desired.
\end{proof}

\subsection{Destabilizing rays on klt pairs}

As in the previous subsection, let $X$ be a compact normal K\"ahler space. Additionally, let
$\Delta\geq0$ be an effective $\mathbb R$-Weil divisor such that
$K_X+\Delta$ is $\mathbb R$-Cartier and $(X,\Delta)$ is klt. Let $\psi$ be a qpsh function on $X$ such that $e^{-\psi} d\mu$ has unit mass, where $d\mu$ is an adapted measure for $(X,\Delta)$.

Throughout this section, we fix a log resolution
$
\pi\colon Y\to X$
of the pair $(X,\Delta)$. As explained in \cite{BBEGZ16}, the klt
condition allows us to write
\[
\pi^*d\mu
=
e^{\psi^+-\psi^-}\omega_Y^n,
\]
where $\omega_Y$ is a K\"ahler form on $Y$ and $\psi^+,\psi^-$ are
quasi-plurisubharmonic functions on $Y$ with divisorial
singularities and the divisorial Lelong numbers of $\psi^-$ are strictly less than $1$. In particular,
$e^{-\psi^-}\in L^p(Y,\omega_Y^n)$ for some $p>1$.

Consequently,
\[
\pi^*\bigl(e^{-\psi}\,d\mu\bigr)
=
e^{\psi^+-\psi^--\psi \circ \pi}\omega_Y^n.
\]
Under the assumed integrability of $e^{-\psi}\,d\mu$, this is a tame
measure on $Y$ in the sense of \cite{BBEGZ16, DZ24}: indeed, $\psi^+$ has
analytic singularity type, while $\psi^-+\pi^*\psi$ is
qpsh.

Pulling back takes $\theta$-psh potentials on
$(X,\Delta)$ to the corresponding $\pi^*\theta$-psh potentials on the smooth K\"ahler
manifold $Y$. Thus all psh data from the singular space $(X,\Delta)$ can be pulled back to  the smooth $Y$, as considered in \cite{DZ24}. Hence all the results from that paper apply word-for-word in our singular context as well.

The radial formulation of the next result builds on the Ross--Witt Nystr\"om correspondence between test curves and weak geodesic rays \cite{RWN14}, together with its extensions to finite-energy rays and big cohomology classes developed in \cite{DDNL3,DX20, DZ24}.

\begin{theorem}[=Theorem \ref{thm:  non_bound_imply_destab_ray}]
\label{thm:  non_bounded_gives_destab_ray} Let $0< \lambda < \delta_\psi(X,\Delta,\{\theta\})$ and assume that $\mathcal{D}^\lambda: \mathcal E^1_\theta \to \Bbb R$ is not bounded from below.  Then there exists a non-trivial finite energy geodesic ray $\{v_t\}\in\mathcal{R}^1_\theta$ such that $\mathcal D^{\lambda+ \varepsilon}\{v_t\} < 0$ for all $0< \lambda  + \varepsilon < \delta_\psi(X,\Delta,\{\theta\})$.
\end{theorem}

\begin{proof}  Since $\mathcal{D}^\lambda$ is not bounded from below there exist potentials $u^j \in\mathcal{E}^1_\theta $ with $\sup_X u^j =0$, such that
$$d_1(V_\theta,u^j) = -I(u^j) \geq j \ \ \ \textup{ and } \ \ \ \mathcal{D}^\lambda(u^j) \leq -j.$$

Indeed, if it were impossible to arrange $d_1(V_\theta,u^j) = -I(u^j) \geq j \to \infty$ as above, then, for sufficiently large $d> 0$, $w \to \mathcal{D}^\lambda(w) $ would be bounded below for $w \in \mathcal E^1_\theta$ satisfying $d_1(V_\theta, w) \geq d$ and $\sup_X w =0$. We now observe that $w \to \mathcal{D}^\lambda(w)$ is also bounded below for $d_1(V_\theta,w) \leq d$ and $\sup_X w =0$, contradicting the non-boundedness from below of $\mathcal{D}^\lambda(w)$. Since $I(\cdot)$ is $d_1$-continuous, it is enough to show that $\mathcal L^\lambda(w)$ is bounded below for $w$ satisfying $d_1(V_\theta,w) \leq d$ and $\sup_X w=0$. The set defined by $d_1(V_\theta,w) \leq d$ and $\sup_X w=0$ is weakly compact in $L^1$ \cite{BBEGZ16}, and therefore Theorem \ref{thm: main_klt} gives the required lower bound for $\mathcal L^\lambda(w)$. The same result also implies that the map $w \to \mathcal D^\lambda(w), \  w \in \mathcal E^1_\theta$ is $d_1$-continuous.

By the definition of $\mathcal{D}^\lambda$ and the inequality $-I(u^j) \geq j$, we obtain
\begin{equation}\label{eq:  L_energy_est_1}
\mathcal{L}^\lambda(u^j) \leq I(u^j) \ \textup{ and }  \ \  -I(u^j)\to \infty .
\end{equation}

Let $[0,-I(u^j)] \ni t \to u^j_t \in \mathcal{E}^1_\theta$ be the weak geodesic segment joining $u^j_0 := V_\theta$ to $u^j_{-I(u^j)} := u^j$.

Since $\sup_X u_t^j=0$, Hartogs' lemma implies the existence of a subgeodesic ray
$[0,\infty) \ni t \to u_t \in \mathcal{E}^1_\theta$ such that $u_0 = V_\theta$ and $u^j(t,x) := u^j_t(x) \to u(t,x) := u_t(x)$ in $L^1_{\textup{loc}}((0,\infty) \times X)$. Since $I_\theta(\cdot)$ is upper semicontinuous in the $L^1$ topology, we obtain 
\begin{equation}\label{eq:  I_est}
-t \leq I(u_t) \leq 0.
\end{equation}
It is possible that the subgeodesic ray $t \to u_t$ is identically equal to $V_\theta$, but we now show that this is not the case.

Let $\tau \in (-1,0)$. Proposition \ref{prop:  L^1_conv_Lagrange_dual} gives $\hat u^j_{\tau} \to_{L^1} \hat u_{\tau}$. Together with \eqref{eq:  L_energy_est_1}, we claim that this convergence implies
\begin{equation}\label{eq:  exponent_est}
\int_X e^{-\lambda \hat u_{\tau}} e^{-\psi}d \mu = \infty.
\end{equation}
Indeed, suppose that $\int_X e^{-\lambda \hat u_{\tau}}e^{-\psi}d \mu < \infty$. Applying Theorem~\ref{thm: main_klt} locally to the quasi-psh functions $\hat u^j_\tau+\lambda^{-1}\psi$ and $\hat u_\tau+\lambda^{-1}\psi$, after adding a common smooth local potential, we obtain
$$\int_X e^{-\lambda  \hat u^j_{\tau}} e^{-\psi}d \mu \to \int_X e^{-\lambda \hat u_{\tau}} e^{-\psi}d \mu < \infty.$$
We will show, however, that this is a contradiction because the terms on the left are unbounded. 

Indeed, by the definition of the Legendre transform $\hat u^j_\tau$, we have $\hat u^j_{\tau} \leq   u^j + \tau I(u^j)$. Thus, we obtain
\begin{flalign*}
\log  \int_X e^{-\lambda \hat u^j_{\tau}} e^{-\psi}d \mu &\geq   \log  \int_X e^{-\lambda (u^j + \tau I(u^j))} e^{-\psi}d \mu = -\lambda \mathcal L^\lambda(u^j) - \lambda \tau I(u^j) \\
&\geq -\lambda(1 + \tau )  I(u^j)\to \infty,
\end{flalign*}
where the last two steps follow from \eqref{eq:  L_energy_est_1} and the fact that $1 + \tau >0$.

Thus, we have proved \eqref{eq:  exponent_est}. This immediately shows that $\{u_t\}$ cannot be identically equal to $V_\theta$, since $\lambda < \delta_\psi(X,\Delta,\{\theta \}) \leq c_{(e^{-\psi}\mu,X)}[V_\theta]$. Using \cite[Theorem 1.4]{DZ24}, \eqref{eq:  exponent_est}, and \eqref{eq:  I_est}, we now obtain
$$\mathcal D^\lambda\{u_t\} = -I\{u_t\} + \sup\left\{s\in\RR:
\int_X e^{-\lambda\hat u_s}\,e^{-\psi}d \mu<\infty
\right\}  \leq  1 + \tau.$$
Since $\tau \in (-1,0)$ is arbitrary, letting $\tau \searrow -1$, we obtain that $\mathcal D^\lambda\{u_t\} \leq 0.$

Let $\{v_t\} \in \mathcal R^1_\theta$ be the maximization of the sublinear finite energy subgeodesic $\{u_t\}$ (recall \cite[(30)]{DZ24}). 

By definition of maximization, $\hat v_\tau = P_\theta[\hat u_\tau], \tau <0$, hence due to \cite[Proposition 2.5]{DZ24} we get that $\int_X e^{-\lambda \hat v_{\tau}} e^{-\psi}d \mu = \infty$ for $\tau \in (-1,0)$. Thus  $\{v_t\}$ cannot be constant equal to $V_\theta$.

By \cite[Lemma 5.5]{DXZ25}  we have that $\mathcal L^\lambda\{v_t\} = \mathcal L^\lambda\{u_t\}$ and $I\{u_t\} = I\{v_t\}$. Thus  $\mathcal D^\lambda\{v_t\}\leq 0$ and $\sup_X v_t =0$.

That $\mathcal D^{\lambda+ \varepsilon}\{v_t\} < 0$ for all small $\varepsilon >0$ follows immediately from the next lemma.
\end{proof}

The following result is a direct consequence of \cite[eq. (41)]{DZ24}:
\begin{lemma}
We fix $\eta > \lambda > 0 $ and $\alpha\in(0,\min\{\lambda,\alpha_\psi(X,\Delta, \{\theta\})\})$. Then there exists a constant $C_\alpha>0$ such that:
\begin{equation} \label{eq:  D_twist_compare}
{\mathcal{D}^\lambda(w)}\geq\frac{\eta(\lambda-\alpha)}{\lambda(\eta-\alpha)} \mathcal{D}^\eta(w)-\frac{\alpha(\eta-\lambda)}{\lambda(\eta-\alpha)} I_\theta(w) - \frac{\eta-\lambda}{\lambda(\eta-\alpha)} C_\alpha, \ \ \ w \in \mathcal{E}^1_\theta, \ \sup_X w =0.
\end{equation}
\end{lemma}

\begin{corollary} Suppose that $0< \lambda < \delta_\psi(X,\Delta,{\theta})$. Then \eqref{eq:  KE_cont eq} has a solution $u_\lambda \in \PSH_\theta$ with minimal singularity.
\end{corollary}

\begin{proof} Let $\lambda' \in (\lambda,\delta_\psi(X,\Delta,{\theta}))$. By Theorem \ref{thm: delta_eq delta_A} we get that $\mathcal D^{\lambda'} :\mathcal E^1_\theta \to \Bbb R$ is bounded from below. The previous lemma then implies that $\mathcal D^\lambda :\mathcal E^1_\theta \to \Bbb R$ is $d_1$-coercive. Since $u \to \mathcal L^\lambda(u)$ is $L^1$-continuous due to Theorem \ref{thm: main_klt}, \cite[Proposition 5.2]{DZ24} implies that \eqref{eq:  KE_cont eq} has a solution $u_\lambda \in \mathcal E^1_\theta$. That $u_\lambda$ has minimal singularity type is a consequence of the Kolodziej type estimate of \cite[Theorem B]{BEGZ10}.
\end{proof}

\section{Local alpha invariants and normalized volumes}
\label{sec:  GT}
Let $(X,p)$ be an isolated log terminal singularity, and let $p \in \Omega \subset X$ be a strongly pseudoconvex neighborhood, as described in the introduction. 

We recall a special case of the  Cegrell class \cite{Ce98} that is needed  in the definition of the local alpha
invariants from \cite{GT23}:
\[
\mathcal T_0(\Omega)
:=
\left\{
u\in\PSH(\Omega)\cap\mathcal C^0(\overline\Omega):
u|_{\partial\Omega}=0,
\quad
\int_\Omega (\ddc u)^n<\infty
\right\}.
\]
Following \cite[Definition~2.6]{GT23}, let $\mathcal F(\Omega)$ be the
set of all $u\in\PSH(\Omega)$ for which there exists a decreasing
sequence $u_j\in\mathcal T_0(\Omega)$ such that
\[
u_j\searrow u,
\qquad
\sup_j\int_\Omega(\ddc u_j)^n<\infty.
\]
For $u\in\mathcal F(\Omega)$, the Monge--Amp\`ere measure
$(\ddc u)^n$ is well defined and has finite mass. We set
\begin{equation}
\label{eq: F1__local}
\mathcal F_1(\Omega)
:=
\left\{
u\in\mathcal F(\Omega):
\int_\Omega(\ddc u)^n\leq 1
\right\}.
\end{equation}
By \cite[Proposition~2.7]{GT23}, relying on an observation of Zeriahi \cite{Zer09}, the class $\mathcal F_1(\Omega)$ is
compact in the $L^1(\Omega)$ topology.

\begin{theorem}[=Theorem~\ref{thm: local__alpha_intro}]
\label{thm: local__alpha}
Let $(X,p)$ be an isolated log terminal singularity as above. Then
\[
\alpha(X,p)
=
\widetilde\alpha(X,p)
=
\widehat{\mathrm{vol}}(X,p)^{1/n}.
\]
\end{theorem}
\begin{proof}
It follows directly from the definitions that
$
\alpha(X,p)\leq\widetilde\alpha(X,p).$
We prove the reverse inequality. We fix
$0<\lambda<\widetilde\alpha(X,p)$ and suppose, toward a
contradiction, that
\[
\sup_{u\in\mathcal F_1(\Omega)}
\int_\Omega e^{-\lambda u}\,d\mu_p=\infty.
\]
We may then choose $u_j\in\mathcal F_1(\Omega)$ such that
$
\int_\Omega e^{-\lambda u_j}\,d\mu_p\to\infty.
$
Due to the approximation property of $\mathcal F_1(\Omega)$ and the monotone convergence theorem, we can assume that $u_j\in\mathcal T_0(\Omega)$.

\begin{equation}
\int_\Omega e^{-\lambda u_j}\,d\mu_p\to\infty.
\end{equation}

By strong pseudoconvexity of $\Omega$, there exists an open $\Omega'' \subset X$ satisfying $\overline{\Omega} \subset \Omega''$ and a global  strongly psh $C^2$ defining function $\rho: \Omega'' \to \Bbb R $ such that $\{\rho  < 0\} = \Omega$.

Let $\delta >0$ be such that  $\Omega' : = \{\rho < \delta\} \subset \Omega''$  is relatively compact. Borrowing ideas from the argument of \cite[Proposition~3.14]{GT23}, we introduce the envelope 
$$v_j : = \sup\{v \in \PSH(\Omega'), v \leq 0, v|_{\Omega} \leq u_j\}.$$
Since $A (\rho - \delta)$ for $A>0$ big enough is a candidate for $v_j$ we get  that $A (\rho - \delta) \leq v_j$. 

Using a standard approximation argument having roots in \cite{BT76}, the Monge--Amp\`ere measure $(\ddc v_j)^n $ is concentrated on the contact set $\{v_j = u_j\}$ and $(\ddc v_j)^n  \leq (\ddc u_j)^n $. We conclude that $v_j \in \mathcal F_1(\Omega')$.

We use the same notation $\mu_p$
for an adapted measure on $\Omega'$ as its normalization does not affect
any of the integrability exponents. 

By the compactness of
$\mathcal F_1(\Omega')$, after passing to a subsequence we may assume
that
$
v_j\to v$ in $L^1(\Omega')$
for some $v\in\mathcal F_1(\Omega')$. Since $v_j\leq u_j$ on $\Omega$ we must have that 
$$\int_\Omega  e^{-\lambda v_j} d\mu_p \to \infty.$$

Now \cite[Proposition 5.8]{GT23} allows us to compute
$\widetilde\alpha(X,p)$ using both $\Omega$ and $\Omega'$. Hence $\lambda<\widetilde\alpha(X,{\Omega'},p)=\widetilde\alpha(X,{\Omega},p)\leq c_{\mu_p,\Omega'}[v]$ and  
we can choose $t$ such that
\[
\lambda<t<c_{\mu_p}(v,{\Omega'}),
\qquad
\int_{\Omega'}e^{-tv}\,d\mu_p<\infty.
\]
But applying Theorem~\ref{thm: main_klt} on $\Omega'$, with
$K=\overline\Omega$ immediately gives
$
e^{-\lambda v_j}\to e^{-\lambda v}$ in $L^1(\Omega,\mu_p).$
 In particular, the integrals
$\int_\Omega e^{-\lambda v_j}\,d\mu_p$ are bounded, a contradiction. This concludes $\lambda\leq\alpha(X,p)$. Letting
$\lambda\nearrow\widetilde\alpha(X,p)$ yields
$\widetilde\alpha(X,p)\leq\alpha(X,p).$

Finally, \cite[Proposition~5.8]{GT23} gives
$
\widetilde\alpha(X,p)
=
\widehat{\mathrm{vol}}(X,p)^{1/n},$
which completes the proof.
\end{proof}

\begingroup
\renewcommand*{\bibfont}{\normalfont\small}
\setlength{\bibitemsep}{0pt}
\setlength{\biblabelsep}{4pt}
\printbibliography
\endgroup

\small
\noindent {\sc Department of Mathematics, University of Maryland}\\
{\tt tdarvas@umd.edu}\vspace{0.1in}

\noindent {\sc School of Mathematical Sciences, Beijing Normal University}\\
{\tt kwzhang@bnu.edu.cn}
\end{document}